%% file: main-2.tex
\documentclass[english]{amsart}
\usepackage[english]{babel}
\usepackage{bbm}
\usepackage{hyperref}
\usepackage{verbatim}
\usepackage{amsfonts,amssymb,mathtools}
\usepackage{graphicx,xcolor}
\graphicspath{{Figures/}}
\usepackage[style=numeric,maxcitenames=2,maxbibnames=10,doi=false,isbn=false,url=false,natbib=true,giveninits=true,hyperref,bibencoding=utf8,backend=biber]{biblatex}
\usepackage{amsmath}
\AtEveryBibitem{%
  \clearfield{note}%
  \clearfield{issn}%
  \clearlist{language}%
  }

\newcommand{\N}{\ensuremath{\mathbb{N}}}
\newcommand{\R}{\ensuremath{\mathbb{R}}}

\newcommand{\Z}{\ensuremath{\mathbb{Z}}}
\newcommand{\E}{\ensuremath{\mathbb{E}}}
\renewcommand{\P}{\ensuremath{\mathbb{P}}}

\newcommand{\ind}[1]{\ensuremath{\mathbbm{1}_{\left\{#1\right\}}}}
\newcommand{\diff}{\mathop{}\mathopen{}\mathrm{d}}
\newcommand{\cal}[1]{\ensuremath{\mathcal{#1}}}
\newcommand\croc[1]{\left\langle #1\right\rangle}
\newcommand\steq[1]{\stackrel{\text{\rm #1.}}{=}}
\newcommand\Dom[1]{\overline{#1}}

\def\eps{\varepsilon}
\def\cadlag{c\`adl\`ag }

\newtheorem{proposition}{Proposition}
\newtheorem{definition}[proposition]{Definition}
\newtheorem{lemma}[proposition]{Lemma}
\newtheorem{theorem}[proposition]{Theorem}

\allowdisplaybreaks

\title[Averaging Principles for Markovian Models of Plasticity]{Averaging Principles for Markovian Models of Plasticity}

\date{\today}
\author[Ph. Robert]{Philippe Robert}
\email{Philippe.Robert@inria.fr}
\urladdr{http://www-rocq.inria.fr/who/Philippe.Robert}
\address[Ph.~Robert, G.~Vignoud]{INRIA Paris, 2 rue Simone Iff, 75589 Paris Cedex 12, France}
\author[G. Vignoud]{Ga\"etan Vignoud${ }^1$}
\email{Gaetan.Vignoud@inria.fr}
\address[G. Vignoud]{Center for Interdisciplinary Research in Biology (CIRB) - Coll\`ege de France (CNRS UMR 7241, INSERM U1050), 11 Place Marcelin Berthelot, 75005 Paris, France}
\thanks{${}^1$Supported by PhD grant of \'Ecole Normale Sup\'erieure, ENS-PSL}

\begin{document}
\begin{abstract}
  Mathematical models of biological neural networks are associated to a rich and complex class of stochastic processes.
%  When the  connectivity of the network is fixed, various stochastic limit theorems, such as mean-field approximation, chaos propagation and renormalization have been used successfully to study the qualitative properties of these networks. 
In this paper, we consider a simple {\em plastic} neural network whose {\em connectivity/synaptic strength} $(W(t))$ depends on a set of activity-dependent processes to model {\em synaptic plasticity}, a well-studied mechanism from neuroscience.
A general class of stochastic models has been introduced in~\cite{robert_mathematical_2020} to study the stochastic process $(W(t))$. It has been observed experimentally that its dynamics occur on  much slower timescale than that of the main cellular processes.  The purpose of this paper is to establish limit theorems for the distribution of $(W(t))$ with respect to the fast timescale of neuronal  processes.

The central result of the paper is an averaging principle for the stochastic process $(W(t))$. Mathematically, the key variable is the point process whose jumps occur at the instants of neuronal spikes. A thorough analysis of several of its unbounded additive functionals is achieved in the slow-fast limit. Additionally, technical results on interacting shot-noise processes are developed  and used in the general proof of the averaging principle.
%A comparison with classical related results of statistical physics in neuroscience is done in~\cite{robert_mathematical_2020}. 
\end{abstract}
\maketitle

 \vspace{-5mm}

\bigskip

\hrule

\vspace{-3mm}

\tableofcontents

\vspace{-1cm}

\hrule

\bigskip
\section{Introduction}
In neuroscience, the encoding of memory is associated with the evolution of neural connectivity in different parts of the brain.
The transmission of neuronal information results from the exchange of a chemical/electrical signal at a synaptic junction, or synapse, where two neighboring neurons interact.
Several experimental studies have shown that the intensity of this connection, or synaptic strength, is not fixed and that its evolution is a necessary component of learning.
Synaptic plasticity refers to the collection of cellular processes that modifies the synaptic weight and has been studied extensively in neurobiology.
More importantly, these mechanisms are modulated in complex ways by the spiking activity of adjacent neurons.

In~\citet{robert_mathematical_2020} we have introduced a general class of mathematical models to represent and study a large class of synaptic plasticity mechanisms.
These models rely on two clearly stated hypotheses: the effect of plasticity is seen on the synaptic strength on {\em long timescales} and it only depends on the {\em relative timing of the spikes}.
This type of plasticity, known as Spike-Timing-Dependent Plasticity (STDP), has been extensively studied in experimental and computational neuroscience, see~\citet{feldman_spike-timing_2012,morrison_phenomenological_2008} for references.
To the best of our knowledge, few rigorous mathematical studies of these models exist, see Section~1.4 of~\cite{robert_mathematical_2020}. As discussed in this reference, measurements show that this system follows {\em slow-fast} dynamics in the sense that the synaptic weight process $(W(t))$ evolves on a slower timescale than neuronal activity associated to the spiking patterns.
The purpose of the current paper is to prove limit theorems of a  scaled version of the corresponding stochastic processes.

\subsection{A Simple Model}
\label{secsec:simplemodel}
We begin by the description of a simplified system to highlight the different components of these stochastic models.
In particular, we will focus on the time evolution of the synaptic strength process $(W(t))$ subject to synaptic plasticity.
It is important to stress here that the synaptic transmission between neural cells is unilateral, in the sense that, the signal goes from an input neuron, called the {\em pre-synaptic neuron}, towards an output neuron, the {\em post-synaptic neuron}.

The stochastic process can be represented by the following variables,
\begin{enumerate}
\item  the membrane potential $X$ of the output cell;
\item the synaptic weight $W$, modeling the strength of the connection from the input neuron to the output neuron.
\end{enumerate}
When the input neuron is spiking, i.e. when it transmits a chemical/electrical signal to the output neuron, {\em a pre-synaptic spike}, the membrane potential $X{=}x$ is updated to $X{=}x{+}w$, where $w$ is the current synaptic weight.

In state $X{=}x$, the output neuron emits a spike at rate $\beta(x)$, {\em a post-synaptic spike},  where $\beta$ is {\em the activation function},
generally taken as a non-decreasing function of the membrane potential.
Accordingly, a pre-synaptic spike and the associated jump in membrane potential leads to and increase of the spiking probability, i.e the post-synaptic neuron tends to spike ``just after'' receiving a pre-synaptic input.
As explained in~\citet{robert_mathematical_2020}, STDP synaptic mechanisms depend, in a complex way, on past spiking times of both adjacent neurons.

More formally, in our simple example, the time evolution is described by the following set of Stochastic Differential Equations (SDEs),
\begin{equation}\label{SimpEqW}
\begin{cases}
\diff X(t) &\displaystyle = {-}X(t)\diff t+W(t)\mathcal{N}_{\lambda}(\diff t),\\
\diff Z(t) &\displaystyle =   {-}\gamma Z(t) \diff t+B_1\mathcal{N}_{\lambda}(\diff t)+B_2\mathcal{N}_{\beta,X}(\diff t),\\
\diff W(t) &\displaystyle = Z(t{-})\mathcal{N}_{\beta,X}(\diff t),
\end{cases}
\end{equation}
where $h(t{-})$ is the left-limit of the function $h$ at $t{>}0$ and, for $i{=}\{1,2\}$, $B_i{\in}\R_+$.
We discuss briefly the random variables involved.
\begin{enumerate}
\item $\mathcal{N}_{\lambda}$ and $\mathcal{N}_{\beta,X}$.\\
These random variables are point processes representing the sequences of spike times of the pre- and post-synaptic neuron.
An instant of ${\cal N}_\lambda(\diff t)$ is associated to a pre-synaptic spike and, as a result leads the increment of the post-synaptic membrane potential $X(t-)$ by $W(t)$.
In the present work, ${\cal N}_\lambda$ is assumed to be a Poisson process with rate $\lambda$.

The point process $\mathcal{N}_{\beta,X}$ is for the sequence of instants of  the post-synaptic spike train.
It is a non-homogeneous Poisson process with (random) intensity function $(\beta(X(t{-})))$. It is formally defined by Relation~\eqref{Nbeta} below. 
\item $(Z(t))$.\\
The process $(Z(t))$ encodes the past spiking activity of both neurons through an additive functional of $\mathcal{N}_{\lambda}$ and $\mathcal{N}_{\beta,X}$ with an exponential decay factor $\gamma{>}0$.
Lemma~\ref{lemma:expofilter} below gives the representation, for $t{\ge}0$,
\[
Z(t)=Z(0){+}B_1\int_0^t e^{-\gamma(t-s)}\mathcal{N}_{\lambda}(\diff s){+}B_2\int_0^t e^{-\gamma(t-s)}\mathcal{N}_{\beta,X}(\diff s).
\]
See~\citet{robert_mathematical_2020} for more details.

\item $(W(t))$.\\
The synaptic weight $W$ is increased at each jump of ${\cal N}_{\beta,X}$ by the value of $(Z(t))$.
\end{enumerate}

From a biological point of view, the relevant process is $(W(t))$, because it describes the time evolution of the synaptic strength, i.e. the intensity of transmission between two connected neurons.
Even if it has been extensively studied both in experimental neuroscience and statistical physics, there are few rigorous mathematical results on the dynamical evolution of $W$.

From a mathematical perspective, the variables $(X(t),Z(t),W(t))$, solutions of SDE~\eqref{SimpEqW} are central to the model.
Still, the point process ${\cal N}_{\beta,X}$ is the key component of the system  since it drives the time evolution of $(Z(t))$ and $(W(t))$ and, consequently, of $(X(t))$.
Most mathematical difficulties resulting from this model are related to asymptotic estimates of linear functionals of ${\cal N}_{\beta,X}$.
It can also be seen as an extension of Hawkes point processes, see Section~\ref{HawSunSec} below.

The scaling approach of this paper follows from the fact that the model can be expressed as a {\em slow-fast} system.
In fact, neuronal processes, associated to the point processes ${\cal N}_{\lambda}$ and ${\cal N}_{\beta,X}$, occur on a timescale which is much faster than the timescale of the evolution of $(W(t))$.
See~\citet{kempter_hebbian_1999} for example and Sections~1 and~4.1 of~\cite{robert_mathematical_2020} for a discussion on this topic.

Using this scaling for the simple model, the SDE~\eqref{SimpEqW} becomes, for $\eps{>}0$,
\begin{equation}
\label{eq:scaledsimple}
\begin{cases}
\diff X_\eps(t) &\displaystyle = {-}X_\eps(t)\diff t/\eps+W_\eps(t)\mathcal{N}_{\lambda/\eps}(\diff t),\\
\diff Z_\eps(t) &\displaystyle =   {-}\gamma Z_\eps(t) \diff t/\eps+B_1\mathcal{N}_{\lambda/\eps}(\diff t)+B_2\mathcal{N}_{\beta/\eps,X_\eps}(\diff t),\\
\diff W_\eps(t) &\displaystyle = Z_\eps(s{-})\eps\mathcal{N}_{\beta/\eps,X_\eps}(\diff s)\diff t.
\end{cases}
\end{equation}
As it can be seen, the variables $(X_\eps(t))$ and $(Z_\eps(t))$ evolve on the timescale $t{\mapsto}t/\eps$, with $\eps$ small, they are {\em fast variables}.
Conversely, the increments of the variable $W$ are scaled with the parameter $\eps$,  the integration of the differential element $\eps\mathcal{N}_{\beta/\eps,X_\eps}(\diff s)$ on a bounded time-interval is $O(1)$.  For this reason, $(W_\eps(t))$ is described as a {\em slow process}.
This is a classical assumption in the corresponding models of statistical physics. Approximations of $(W_\eps(t))$ when $\eps$ is small are discussed and investigated with ad-hoc methods, see~\citet{kempter_hebbian_1999} for example.

\subsection{Averaging Principles}\label{AVGSubSec}
The main goal of the present paper is to establish a limit result, or averaging principle, for $(W_\eps(t))$ when $\eps$ goes to $0$ for a general class of synaptic plasticity models.

In particular, the averaging principle for the simple model can be expressed as follows.
Under appropriate conditions, if $\Pi_w$ denotes the equilibrium distribution of the variables $(X^w,Z^w)$ when the process $(W(t))$ is constant and equal to $w$, then there exists $S_0{\in}(0,{+}\infty]$, such that the processes $(W_\eps(t),0{\le}t{<}S_0)$  is tight for the convergence in distribution when $\eps$ goes to $0$, and any limiting point $(w(t),0{\le}t{<}S_0)$ satisfies the  following  integral equation,
\begin{equation}\label{ODESimp}
  w(t)= w(0){+} \int_0^t \int_{\R_+^2}z\beta(x)\Pi_{w(s)}(\diff x,\diff z)\,\diff s, \quad t{\in}[0,S_0).
\end{equation}
See Chapter~7 of~\citet{freidlin_random_1998} and~\citet{papanicolalou_martingale_1977} for an introduction to averaging principles in the case of diffusions and~\citet{karatzas_averaging_1992} for jump processes.

\subsection*{Remarks} We quickly discuss several aspects of these results. 
\begin{enumerate}
\item {\sc Tightness Properties of Additive Functionals.}\\
%Convenient ergodicity properties of the fast process $(X^w(t),Z^w(t))$ with $W=w$ fixed are established in Section~\ref{section:invariant}.
The technical problems of this paper essentially lie in the tightness properties of linear functionals of the fast process occupation measures.
Several technical results are developed to alleviate these difficulties: in Section~\ref{ShotSec} of Appendix for properties of interacting shot-noise processes, in Section~\ref{OccSec} on occupation measures for bounded synaptic weights and, finally in Section~\ref{section:proof} for the final tightness results.
The scarcity of rigorous results in the literature may be explained by such obstacles. The reference~\citet{helson_new_2018} for the time-elapsed model is one of the rare examples of rigorous analysis, but it mainly considers bounded updates.

The main difficulty originates, as it could be expected, from the scaled point process $\eps\mathcal{N}_{\beta/\eps,X_\eps}(\diff s)$ associated to post-synaptic spikes and, more precisely, from the tightness of
\[
\left(\int_0^t Z_\eps(s)\eps\mathcal{N}_{\beta/\eps,X_\eps}(\diff s)\diff s\right).
\]
If the model was limited to bounded functionals of the occupation measure, the proof of this tightness property would be significantly simpler, much in the spirit of~\citet{karatzas_averaging_1992}.

\item {\sc Uniqueness}.\\
  If Relation~\eqref{ODESimp}  has a unique solution for a given initial state, a result for the convergence in distribution of $(W_\eps(t))$  when $\eps$ goes to $0$ is therefore obtained. Uniqueness holds if the integrand, with respect to $s$, of the right-hand side of Relation~\eqref{ODESimp} is locally Lipschitz as a function of $w(s)$.
  Regularity properties of the invariant distribution $\Pi_w$ as a function of $w$ need to be verified and this is not a concern in the case of our simple model. We will consider in fact much more general models  for $(X^w,Z^w)$, when $Z^w$ a multi-dimensional process in particular. We did not try to state a set of conditions that can ensure the desired regularity properties  of the corresponding $\Pi_w$. The proof of the Harris ergodicity of $(X^w,Z^w)$ for a fixed $w$ of Section~\ref{section:invariant} of Appendix, though not really difficult, is already cumbersome.

The proof of Proposition~\ref{LemLipSimp}  for the simple model gives an example of how this property can be established. In a general context, this kind of result is generally proved via  the use of a common Lyapounov function for $(X^w,Z^w)$ for all $w$ is in the neighborhood of some $w_0{>}0$.  See~\citet{Khasminski}, for example. Uniqueness results have already been obtained in Sections~5 and~6 of~\citet{robert_mathematical_2020} for several important practical cases. In Section~\ref{sec:simplemodel} we investigate these questions for our simple model.

\medskip

\item {\sc Blow-up Phenomenon.}\\
  The convergence properties are stated on {\em a fixed time interval} $[0,S_0)$.
  For some models, the variable $S_0$ cannot be taken as $+\infty$, see the example of Section~\ref{section:theorems} and Proposition~\ref{ODESimpAvProp}.
  More specifically, the limit in distribution of  $(W_\eps(t))$ as $\eps$ goes to $0$ blows-up, i.e. hits infinity in finite time.
  An analogue property holds for some mathematical models of large populations of neural cells with fixed synaptic strengths.
  See~\citet{caceres2011analysis} for example, where the blow-up phenomenon is the result of mutually exciting dynamics of populations of neural cells. In our case, the strengthening of the connection may grow without bounds when the activation function $\beta$ has a linear growth. See Proposition~\ref{ODESimpAvProp} of Section~\ref{sec:simplemodel}. 
\end{enumerate}

\subsection{A Brief Description of the General Model}\label{QDGM}
We shortly describe the general setting of the models investigated in this paper.
See Section~\ref{ModSec} for a detailed presentation.
\begin{enumerate}
\item The process $(X(t))$.\\
The output neuron follows leaky-integrate dynamics as in Equation~\eqref{SimpEqW}.
In addition, the influence of a post-synaptic spike ${\cal N}_{\beta,X}$ at time $t{>}0$ is represented as a drop $-g(X(t{-}))$ of the post-synaptic potential after the spike;
\item The process $(Z(t)){=}(Z_i(t))$ is a multi-dimensional process satisfying the same type of ODE as in our simple case but with the constants $B_1$ and $B_2$ being replaced by functions $k_1$ and $k_2$ of $Z(t)$. A constant drift term $k_0$ is also added to the dynamics. The $i$th component $(Z_i(t))$ satisfies an SDE of the type
  \[
  \diff Z_i(t) \displaystyle =   ({-}\gamma_i Z_i(t){+}k_{0,i})\diff t{+}k_{1,i}(Z(t{-}))\mathcal{N}_{\lambda}(\diff t){+}k_{2,i}(Z(t{-}))\mathcal{N}_{\beta,X}(\diff t).
  \]
\item Evolution of $(W(t))$. \\
  The dependence is more sophisticated since it involves two additional processes $(\Omega_p,\Omega_d)$. The first one, $(\Omega_p(t))$ integrates, with an exponential decay $\alpha$ a linear combination of the processes leading to potentiation, i.e. to increase the synaptic weight. The process  $(\Omega_d(t))$ has a similar role for depression, i.e. to decrease the synaptic weight.
  They are expressed as, for $a{\in}\{p,d\}$,
\[
    \Omega_a(t){=} \hspace{-1mm}\int_0^t\hspace{-2mm} e^{-\alpha(t-s)}\left[  n_{0,a}(Z(s))\diff s{+}n_{1,a}(Z(s{-}))\mathcal{N}_{\lambda}(\diff s){+}n_{2,a}(Z(s{-}))\mathcal{N}_{\beta,X}(\diff s)\right].
\]
The changes of $(Z(t))$ are thus integrated ``smoothly'' in the evolution of $(W(t))$ in agreement with measurements of the biological literature. See Appendix~A of~\cite{robert_mathematical_2020}.
Finally, $(W(t)){\in}K_W$ verifies
\[
\diff W(t)= M\left(\Omega_p(t),\Omega_d(t),W(t)\right)\diff t,
\]
where $K_W{\subset \R}$ represents the synaptic weight domain, and the functional $M$ is such that $W(t)$ stays in $K_W$ for all $t{\geq}0$.
\end{enumerate}

It has been shown in Section~3 of~\cite{robert_mathematical_2020} that these models encompass most classical STDP models from statistical physics.
The multiple coordinates of $(Z(t))$ can be interpreted as the concentrations of chemical components implicated in plasticity, that are created/suppressed by spiking mechanisms

\subsection{Links to Non-Linear Hawkes Point Processes}\label{HawSunSec}
The spiking instants of a neuron can also be seen as a self-exciting point process since its instantaneous jump rate depends on  past instants of its jumps. 
More formally, this corresponds to the class of Hawkes point process ${\cal M}$ on $\R_+$ associated to a function $\phi$ and exponential decay $\gamma$.
More precisely, it is a non-homogeneous Poisson point process ${\cal M}$ whose intensity function $(\lambda(t))$ is given by
\[
\left(\lambda(t)\right)=\left(\phi\left(\int_0^t e^{-\gamma(t-s)}{\cal M}(\diff s)\right)\right).
\]
These processes have received of lot of attention from the mathematical literature, for some time now.
They are mainly used in models of mathematical finance, but also in neurosciences.
See the pioneering works of~\citet{HawkesOak} and~\citet{Kerstan}.

A special case of the first equation of Relation~\eqref{SimpEqW} is, for $w{\ge}0$,
\[
\diff X(t)  =   {-}X(t) \diff t{+}w{\cal N}_{\lambda}(\diff t){-}\mathcal{N}_{\beta,X}(\diff t),
\]
if $X(0){=}0$,  Lemma~\ref{lemma:expofilter} below  gives the representation
\[
X(t)=w\int_0^t e^{-(t-s)}\mathcal{N}_{\lambda}(\diff s){-} \int_0^t e^{-(t-s)} \mathcal{N}_{\beta,X}(\diff s), \forall t{\ge}0.
\]
Hence, ${\cal N}_{\beta,X}$ can be seen as an extended Hawkes process with activation function $\beta$ and exponential decay $1$.

In the system of equations~\eqref{SimpEqW}, $(X(t))$ and $(Z(t))$ can also be represented as a multi-dimensional Hawkes processes. See~\citet{hawkes1971spectra}.
However in our model, an important feature not present in studies of Hawkes processes has been added: the synaptic weight process $(W(t))$ is not constant.

\subsection{Organization of the Paper}\label{subsec:organisation}
In Section~\ref{ModSec}, the main processes and definitions are introduced as well as assumptions to prove an averaging principle.  The scaling is presented in Section~\ref{ScaledSec} and the averaging principle in Section~\ref{section:theorems}. In this section the general strategy for the proof of the main theorem is detailed.  Section~\ref{section:coupling} investigates monotonicity properties and a coupling result, crucial in the proof of tightness, is proved. Section~\ref{OccSec} is devoted to the tightness of occupation measures of fast processes when the process $(W_\eps(t))$ is assumed to be bounded.  Finally, the proof of the main theorem is completed in Section~\ref{section:proof}.
In Section~\ref{ShotSec} of Appendix, several useful tightness results are proved for interacting shot-noise processes.  The ergodicity properties of fast processes are analyzed in Section~\ref{section:invariant} of Appendix. Section~\ref{app:alternative} of the Appendix discusses averaging principles for related discrete models of synaptic plasticity.
\section{The Stochastic Model}\label{ModSec}
We  define the stochastic model associated to Markovian plasticity kernels introduced in~\citet{robert_mathematical_2020}.
The probabilistic setting of these models along with formal definitions are detailed in the following section.

\subsection{Definitions and Notations}\label{defNotSec}
The space of Borelian subsets of a topological space $H$, is denoted as ${\cal B}(H)$.
Let $(\Omega,{\cal F}, ({\cal F_t}), \mathbb{P})$ be a filtered probability space.
We assume that two independent Poisson processes, ${\cal P}_1$ and ${\cal P}_2$  on $\R_+^2$, with intensity $\diff x{\times}\diff y$ are defined on $(\Omega,{\cal F}, ({\cal F_t}), \mathbb{P})$. 
See \citet{kingman_poisson_1992} for example.
For ${\cal P}{\in}\{{\cal P}_1,{\cal P}_2\}$ and $A$, $B{\in}{\cal B}(\R_+)$ and  a Borelian function $f$  on $\R_+$,
\[
  {\cal P}\left(A{\times}B\right){\steq{def}} \int_{A{\times}B} {\cal P}(\diff x,\diff y),
  \int_{\R_+} f(y){\cal P}(A,\diff y){\steq{def}}   \int_{A{\times}\R_+} f(y){\cal P}(\diff x,\diff y).
\]
For $t{\ge}0$, the $\sigma$-field ${\cal F}_t$ of the filtration $({\cal F}_t)_{t{\ge}0}$ is assumed to contain all events before time $t$ for both point processes, i.e.
\begin{equation}\label{Ft}
  \sigma\left< \rule{0mm}{4mm}\mathcal{P}_1\left(\rule{0mm}{3mm}A{\times}(s,t]\right), \mathcal{P}_2\left(\rule{0mm}{3mm} A{\times}(s,t]\right), A {\in}\mathcal{B}\left( \mathbb{R}_+ \right), s{\leq}t \right>\subset {\cal F}_t.
\end{equation}
A stochastic process $(H(t))$ is {\em adapted} if, for all $t{\ge}0$, $H(t)$ is ${\cal F}_t$-measurable.
It is a {\em \cadlag process} if, almost surely, it is right continuous and has a left limit at every point $t{>}0$, $H(t{-})$ denotes the left limit of $(H(t))$ at $t$.  The Skorohod space of \cadlag functions from $[0,T]$ to $S$ is denoted as $\mathcal{D}([0,T],S)$. See~\citet{billingsley_convergence_1999} and~\citet{ethier_markov_2009}.
The mention of adapted stochastic processes, or of martingale, will be implicitly associated to the filtration $({\cal F}_t)_{t{\ge}0}$.

The set of real continuous bounded functions on the metric space $\mathcal{S}{\subset}\R^d$ is denoted by $\mathcal{C}_b(\mathcal{S})$.
$\mathcal{C}_b^k(\mathcal{S}){\subset}\mathcal{C}_b(\mathcal{S})$ is the set of bounded, $k$-differentiable functions on $\mathcal{S}$ with respect to each coordinate, with all derivatives bounded and continuous.
The multi-dimensional extensions to $\mathcal{S}$  are denoted by $\mathcal{C}_b^k(\mathcal{S},\mathcal{S})$.

\bigskip
\noindent
{\sc Dynamics of a Neuron as an Inhomogeneous Poisson Process}.
We introduce an important point process  ${\cal N}_{\phi,H}$, that represents the spike times of a neuron whose membrane potential process is $(H(t))$, with activation function $\phi$.
$\phi$ is a non-negative \cadlag function on $\R$,  it is defined by
\begin{equation}\label{Nbeta}
  \int_{\R_+}f(u){\cal N}_{\phi,H}(\diff u)\steq{def}\int_{\R_+}f(u){\cal P}_2\left(\rule{0mm}{4mm}\left(\rule{0mm}{3mm}0,\phi(H(u{-}))\right],\diff u\right),
\end{equation}
for any Borelian function $f$ on $\R_+$.

\subsection{The Plasticity Process}
\begin{definition}[Time Evolution]\label{DefPar}
The \cadlag process
  \[
(U(t)){=}(X(t),Z(t),\Omega_p(t),\Omega_d(t),W(t))\in \R{\times}\R_+^{\ell}{\times}\R_+^2{\times}K_W,
  \]
  is solution of the following Stochastic Differential Equations (SDE), starting from some initial state $U(0){=}U_0{=}(x_0,z_0,\omega_{0,p},\omega_{0,d},w_0)$. 
\begin{equation}\label{eq:markov}
\begin{cases}
    \diff X(t) &\displaystyle = {-}X(t)\diff t{+}W(t)\mathcal{N}_{\lambda}(\diff t){-}g\left(X(t{-})\right)\mathcal{N}_{\beta,X}\left(\diff t\right),\\
    \diff Z(t) &\displaystyle =   ({-}\gamma{\odot}Z(t){+}k_0)\diff t{+}k_1(Z(t{-}))\mathcal{N}_{\lambda}(\diff t){+}k_2(Z(t{-}))\mathcal{N}_{\beta,X}(\diff t),\\
  \diff \Omega_a(t)&\displaystyle = {-}\alpha\Omega_a(t)\diff t{+}n_{a,0}(Z(t))\diff t\\&\hspace{1cm}{+}
   n_{a,1}(Z(t{-}))\mathcal{N}_{\lambda}(\diff t){+}n_{a,2}(Z(t{-}))\mathcal{N}_{\beta,X}(\diff t),\quad a{\in}\{p,d\},\\
    \diff W(t) &\displaystyle = M\left(\Omega_p(t),\Omega_d(t),W(t)\right)\diff t,
\end{cases}\end{equation}
with the notation $a{\odot}b{=}(a_kb_k)$ for the {\em Hadamard} product, for $a{=}(a_k)$, $b{=}(b_k){\in}\R_+^\ell$.
\end{definition}
Recall that, see Section~\ref{QDGM},  $K_W$ is an interval of $\R$ which contains the range of values of synaptic weight.

We now state the assumptions used for the proof of Theorem~\ref{theorem:homog}.

\subsubsection{Pre-Synaptic Spikes}\label{SecPre} The instants of pre-synaptic spikes are given by a Poisson process with rate $\lambda{>}0$,
\begin{equation}\label{Nlambda}
    {\cal N}_\lambda(\diff t) \steq{def} {\cal P}_1\left((0,\lambda],\diff t\right),
\end{equation}
where ${\cal P}_1$ is the Poisson point process introduced in Section~\ref{defNotSec}. 
%%%

\subsubsection{Post-Synaptic Spikes}\label{SecPost} 
When the post-synaptic membrane potential is $x$, a post-synaptic spike occurs at rate $\beta(x)$ and leads to a decrease of the membrane potential $x{-}g(x)$.
\begin{itemize}
    \item It is assumed that $\beta$ is a non-negative, continuous function on $\R$ and that $\beta(x){=}0$ for $x{\le}{-}c_\beta{\le}0$.
  Additionally, there exists  a constant  $C_\beta{\geq}0$ such that
\begin{equation}\label{Condbeta}
  \beta(x) {\leq} C_{\beta}(1{+}|x|),\quad \forall x{\in}\R.
\end{equation}
\item The function $g$ is continuous on $\R$ and $0{\le}g(x){\le} \max(c_g,x)$ holds for all $x{\in}\R$, for some $c_g{\ge}0$.
  \end{itemize}
The instants of post-synaptic spikes are represented by the point process $\mathcal{N}_{\beta,X}$.
Recall that
\[
{\cal N}_{\beta,X}(\diff t){=}{\cal P}_2\left(\rule{0mm}{4mm}\left(\rule{0mm}{3mm}0,\beta(X(t{-}))\right],\diff t\right).
\]

\subsubsection{The Process $(Z(t))$}
The process $(Z(t))$ is a multi-di\-men\-sional process, with values in $\R_+^\ell$, it is driven by the general spiking activity of the system, and therefore, depends only on the point processes $\mathcal{N}_{\lambda}$ and $\mathcal{N}_{\beta, X}$. For some models it describes the time evolution of chemical components within the synapse.
\label{SecZ}
$(Z(t))$ is a \cadlag function with values in $\R_+^\ell$, solution of the stochastic differential equation
\begin{equation}\label{KerZ}
    \diff Z(t) =({-}\gamma\odot Z(t){+}k_0)\diff t +k_1(Z(t{-})){\cal N}_{\lambda}(\diff t)+k_2(Z(t{-})){\cal N}_{\beta,X},
\end{equation}
$a{\odot}b{=}(a_i{\times}b_i)$ if $a{=}(a_i)$ and $b{=}(b_i)$ in $\R_+^\ell$, $k_0{\in}\R_+^\ell$ is a constant, $k_{1}$ and $k_{2}$ are measurable functions from $\R_+^\ell$ to $\R^\ell$.
Furthermore, the $(k_i)$ are chosen such that $(z(t))$ has values in $\R_+^{\ell}$ whenever $z(0){\in}\R_+^{\ell}$.

It is assumed that
\begin{enumerate}
\item All coordinates of the vector $\gamma$ are positive;
\item The non-negative functions $k_{i}$, $i{=}\{0,1,2\}$, are ${\cal C}^1_b(\R_+^\ell, \R_+^\ell)$ and bounded  by $C_k{\geq}0$.
\end{enumerate}

\subsubsection{The Process $(\Omega_p(t),\Omega_d(t))$}\label{OmCond}

These variables, in $\R_+^2$ encode, with an exponential decay, the total memory of instantaneous plasticity processes represented by the process $(Z(t))$.
  The process $(\Omega_p(t))$ is driving  potentiation of the synapse, i.e. the derivative of synaptic weight is an increasing function of this variable.
  In an analogous way, $(\Omega_d(t))$ is associated to depression, i.e. the derivative of synaptic weight is a decreasing function of this variable.
The system of equations for $(\Omega_p(t),\Omega_d(t))$ is a set of two one-dimensional SDEs, for $a{\in}\{p,d\}$,
\begin{multline*}
  \diff \Omega_a(t)\displaystyle = {-}\alpha\Omega_a(t)\diff t{+}n_{a,0}(Z(t))\diff t
       \\ {+}  n_{a,1}(Z(t{-}))\mathcal{N}_{\lambda}(\diff t){+}n_{a,2}(Z(t{-}))\mathcal{N}_{\beta,X}(\diff t).
\end{multline*}
        We suppose that there exists a constant $C_n$ such that, for $j{\in}\{0,1,2\}$, $a{\in}\{p,d\}$, $n_{a,j}$ verifies,
\begin{equation}\label{Condn}
  n_{a,j}(z){\le} C_n(1{+}\|z\|),
\end{equation}
where, for $z{\in}\R_+^\ell$,  $\|z\|{=}z_1{+}\cdots{+}z_\ell$.

For any $w{\in}K_W$ and   $a{\in}\{p,d\}$, the discontinuity points of
  \[
  (x,z){\mapsto}(n_{a,0}(z),n_{a,1}(z), \beta(x)n_{a,2}(z))
  \]
 are negligible for the invariant probability distribution $\Pi_w$ of $(X(t),Z(t))$ when $(W(t))$ is constant equal to $w$. See Section~\ref{section:invariant}.

When $(W(t))$ is constant, the process $(X(t),Z(t))$ can be seen as generalized shot-noise processes, see Section~\ref{ShotSec}.
It is well-known that the invariant distribution of the classical, one-dimensional, shot-noise process  is absolutely continuous w.r.t Lebesgue's measure.  See examples of Sections~5 and~6 of~\citet{robert_mathematical_2020} and the reference~\citet{Beznea} for criteria in this domain.

\subsubsection{Dynamics of Synaptic Weight} \label{MCond}
The functional $M$ drives the dynamics of the synaptic weight, the corresponding equation is given by Relation~\eqref{eq:markov}.
In particular, for any $w{\in}K_W$ and any \cadlag piecewise-continuous functions $h_1$ and $h_2$ on $\R_+$, the ODE
\begin{equation}\label{ODEM}
  \frac{\diff w}{\diff t}(t){=}M(h_1(t),h_2(t),w(t)) \text{ with } w(0){=}w,
\end{equation}
for all points of continuity of $h_1$ and $h_2$, has a unique continuous solution denoted by $(S[h_1,h_2](w,t))$ in $K_W$.
We assume that $M$ can be decomposed as
$M(\omega_p,\omega_d,w){=}M_p(\omega_p,w){-}M_d(\omega_d,w) - \delta w$,
where $M_{a}(\omega_a,w)$ is non-negative continuous function, non-decreasing on the first coordinate for a fixed $w{\in}K_w$, and,
\[
   M_a(\omega_a,w)\leq C_M(1+\omega_a),
\]
for all $w{\in}K_W$, for $a{\in}\{p,d\}$.

\subsection{Discrete Models of Synaptic Plasticity}
A model of plasticity with discrete state space has been introduced in~\citet{robert_mathematical_2020}.
The proof of the associated averaging principles for the continuous case can be adapted to such systems.
Relevant parts of the proof are briefly presented in Section~\ref{app:alternative} of the Appendix.
\section{The Scaled Process}\label{ScaledSec}
The SDEs of Definition~\ref{DefPar} are difficult to study without any additional hypothesis.
Existence and uniqueness of solutions to this system are guaranteed by Proposition~1 of~\citet{robert_mathematical_2020}.
It can be seen as an intricate fixed point equation for the processes $(X(t),W(t))$ involving functionals of these processes like ${\cal N}_{\beta,X}$ defined by Relation~\eqref{Nbeta}.

As explained in Section~4 of~\cite{robert_mathematical_2020}, $(X(t),Z(t))$ are associated to fast dynamics at the cellular level while the process $(W(t))$ evolves on a much longer timescale.
For this reason, a scaling parameter $\eps{>}0$ is introduced so that $(X(t),Z(t))$ evolves on the timescale $t{\mapsto}t/\eps$.
More precisely,
\begin{itemize}
\item Fast Processes: $(X(t))$ and $(Z(t))$.\\
The point processes associated to pre- and post-synaptic spikes driving the time evolution of $(X(t))$ and $(Z(t))$ are sped-up by a factor $1/\eps$: ${\cal N}_\lambda{\to}{\cal N}_{\lambda/\eps}$ and ${\cal N}_{\beta,X}{\to}{\cal N}_{\beta/\eps,X}$.
The deterministic part of the evolution is changed accordingly $\diff t{\to}\diff t/\eps$.
\item Slow Processes: $(W(t))$ and $(\Omega_p(t),\Omega_d(t))$.\\
Update of $(\Omega_p(t))$ and $(\Omega_d(t))$ due to fast jump processes have a small amplitude,  ${\cal N}_\lambda{\to}\eps{\cal N}_{\lambda/\eps}$ and ${\cal N}_{\beta,X}{\to}\eps{\cal N}_{\beta/\eps,X}$.
\end{itemize}
Formally, we define the scaled process $(U_\eps(t)){=}(X_\eps(t),Z_\eps(t),\Omega_{\eps,p}(t),\Omega_{\eps,d}(t),W_\eps(t))$, the evolution equations of Definition~\ref{DefPar} become
\begin{align}
    \diff & X_\eps(t) \displaystyle = {-}X_\eps(t)\diff t/\eps{+}W_\eps(t)\mathcal{N}_{\lambda/\eps}(\diff t){-}g\left(X_\eps(t{-})\right)\mathcal{N}_{\beta/\eps,X_\eps}\left(\diff t\right),\label{eq:markX}\\
    \diff & Z_\eps(t) \displaystyle =    \left({-}\gamma{\odot}Z_\eps(t){+}k_0\right)\diff t/\eps{+}k_1(Z_\eps(t{-}))\mathcal{N}_{\lambda/\eps}(\diff t)\label{eq:markZ} \\&\hspace{6cm}+k_2(Z_\eps(t{-}))\mathcal{N}_{\beta/\eps,X_\eps}(\diff t),\notag\\
    \diff &\Omega_{\eps,a}(t)\displaystyle = {-}\alpha\Omega_{\eps,a}(t)\diff t{+}n_{a,0}(Z_{\eps}(t))\diff t\label{eq:markOm}\\&\hspace{1cm}+
    \eps \left(n_{a,1}(Z_{\eps}(t{-}))\mathcal{N}_{\lambda/\eps}(\diff t){+}n_{a,2}(Z_{\eps}(t{-}))\mathcal{N}_{\beta/\eps,X_\eps}(\diff t)\right), \quad a{\in}\{p,d\}\notag\\
    \diff & W_\eps(t) \displaystyle = M\left(\Omega_{\eps,p}(t),\Omega_{\eps,d}(t),W(t)\right) \diff t.\label{eq:markW}
\end{align}
For simplicity, the initial condition of $(U_\eps(t))$ is assumed to be constant,
\begin{equation}\label{InCond}
  U_\eps(0) = U_0 = (x_0,z_0,\omega_{p,0},\omega_{d,0},w_0).
\end{equation}
Some simplifications of this (heavy) mathematical framework can be expected when $\eps$ goes to $0$.
We first introduce the notion of fast variables which correspond to the processes $(X(t),Z(t))$ with the synaptic weight process $(W(t))$ taken as constant.

\subsection{Fast Processes}
\begin{definition}\label{DefFastw}
  For $w{\in}K_W$,  $(X^w(t),Z^w(t))$ is the Markov process in $\R{\times}\R^\ell_+$ defined by the SDEs
\[
    \begin{cases}
\displaystyle \diff X^w(t) &= \displaystyle{-}X^w(t)\diff t{+}w\mathcal{N}_{\lambda}(\diff t){-}g\left(X^w(t{-})\right)\mathcal{N}_{\beta,X^w}\left(\diff t\right),\\
\displaystyle  \diff Z^w(t) &= \displaystyle \left({-}\gamma{\odot}Z^w(t){+}k_0\right)\diff t\\
               &\displaystyle\hspace{2cm} {+}k_1(Z^w(t{-}))\mathcal{N}_{\lambda}(\diff t){+}k_2(Z^w(t{-}))\mathcal{N}_{\beta,X^w}(\diff t).
    \end{cases}
\]
\end{definition}
Let $f{\in}\mathcal{C}_b^1(\R{\times}\R_+^\ell)$, then, with Equations~\eqref{eq:markX} and~\eqref{eq:markZ}, we have that
\[
  (M^F_{f,\eps}(t))\steq{def}\left(f(X_\eps(t),Z_\eps(t)){-}f(x_0,z_0){-}\frac{1}{\eps}\int_0^tB^F_{W_\eps(s)}(f)(X_\eps(s),Z_\eps(s))\diff s\right)
\]
is a local martingale, where, for $v{=}(x,z){\in}\R{\times}\R_+^{\ell}$ and,
\begin{multline}\label{BFGen}
B^F_w(f)(v)\steq{def}{-}x\frac{\partial f}{\partial x}(x,z){+}\croc{{-}\gamma{\odot}z+k_0,\frac{\partial f}{\partial z}(x,z)}
\\+\lambda \left(\rule{0mm}{4mm}f(x{+}w,z{+}k_1(z)){-}f(u)\right)
+\beta(x)\left(\rule{0mm}{4mm}f(x{-}g(x),z{+}k_2(z)){-}f(u)\right),
\end{multline}
with
\[
\frac{\partial f}{\partial z}(x,z){=}\left(\frac{\partial f}{\partial z_i}(x,z), i{\in}\{1,\ldots,\ell\}\right).
\]
$B^F_w$ is called the infinitesimal generator of the fast processes $(X^w(t),Z^w(t))$.

In Proposition~\ref{InvPropFP} of Appendix~\ref{section:invariant}, we establish that, under the conditions of Sections~\ref{SecPost} and~\ref{SecZ}, the fast process $(X^w(t),Z^w(t))$ has a unique invariant distribution $\Pi_w$.
\subsection{Functionals of the Occupation Measure}
We start with a rough, non-rigorous, picture of results that are usually established for {\em slow-fast} systems.
\begin{definition}[Occupation Measure]\label{OccMeas}
The occupation measure  is the non-negative measure $\nu_\eps$ on $[0,T]{\times}\R{\times}\R_+^\ell$ such that
\begin{equation}\label{OddMDef}
\nu_\eps(G)\steq{def}\int_{[0,T]{\times}\R{\times}\R_+^\ell}\hspace{-3mm} G(s,x,z)\nu_\eps(\diff s,\diff x,\diff z)\steq{def}\int_{[0,T]} G(s,X_\eps(s),Z_\eps(s))\diff s.
\end{equation}
for any non-negative Borelian function $G$ on $[0,T]{\times}\R{\times}\R_+^\ell$.
\end{definition}

The integration of Relation~\eqref{eq:markOm} gives the identity, for $a{=}\{p,d\}$
\begin{multline*}
\Omega_{\eps,a}(t) = \omega_{0,a} {-}\alpha\int_0^t \Omega_{\eps,a}(s)\diff s{+}\int_0^tn_{a,0}(Z_{\eps}(s))\diff s\\{+}  \int_0^t  n_{a,1}(Z_{\eps}(s{-}))\eps\mathcal{N}_{\lambda/\eps}(\diff s){+}\int_0^t  n_{a,2}(Z_{\eps}(s{-}))\eps\mathcal{N}_{\beta/\eps,X_\eps}(\diff s).
\end{multline*}

An averaging principle is said to hold when the convergence in distribution
\begin{multline}\label{eqcvp}
\lim_{\eps\to 0} \left(\int_0^t G(X_\eps(s),Z_\eps(s))\diff s\right)=\lim_{\eps\to 0} \left(\int_{\R{\times}\R^\ell_+} G(x,z)\nu_{\eps}(\diff s, \diff x, \diff z)\right)\\=\left(\int_0^t \int_{\R{\times}\R^\ell_+} G(x,z)\Pi_{w(s)}(\diff x,\diff z) \diff s\right),
\end{multline}
holds for a sufficiently rich class of Borelian functions $G$.
%The integral on the left-hand side of the above relation is the {\em occupation measure} of the fast process, defined above.
Usually, it is enough to prove the weak convergence of the occupation measure for bounded Borelian functions $G$.

In our case, there are  important examples where $G$ has a linear growth with respect to the coordinates $x$ or $z{=}(z_j)$.
Additionally, convergence results in distribution of the jump processes such as,
\[
\lim_{\eps\to 0} \left(\int_0^t G(X_\eps(s),Z_\eps(s))\eps{\cal N}_{\lambda/\eps}(\diff s)\right)=\left(\lambda \int_0^t \int_{\R{\times}\R^\ell_+} G(x,z)\Pi_{w(s)}(\diff x,\diff z) \diff s\right),
\]
and
\[
\lim_{\eps\to 0} \left(\int_0^t  \hspace{-1mm}G(X_\eps(s),Z_\eps(s))\eps{\cal N}_{\beta/\eps,X_\eps}(\diff s)\right)=\left(\int_0^t \int_{\R{\times}\R^\ell_+} \hspace{-5mm}G(x,z)\beta(x)\Pi_{w(s)}(\diff x,\diff z) \diff s\right)
\]
are also required. They are not  straightforward consequences of Relation~\eqref{eqcvp} as it is usually the case for bounded $G$.
%Moreover, these results are established using the compensating martingale, which need to tends to $0$ to have the results/
See~\citet{karatzas_averaging_1992} for example.
These technical difficulties have to be overcome to establish the tightness of the processes $(\Omega_\eps(t))$, and consequently of $(W_\eps(t))$.
As a result, additional limit results have to be established at this point, see Section~\ref{OccSec}.
Furthermore, as $\eps$ goes to $0$, the process $(\Omega_{\eps,p}(t),\Omega_{\eps,d}(t),W_{\eps}(t))$ should converge to a process $(\omega_p(t),\omega_d(t),w(t))$ satisfying the relation,
\[
\begin{cases}
\displaystyle\omega_a(t) = \omega_{a,0} {-}\alpha\int_0^t \omega(s)\diff s{+}\int_0^t\int_{\R^\ell_+} n_{a,0}(z)\Pi_{w(s)}\left(\R_+,\diff z\right) \diff s
\\\hspace{1mm}{+} \displaystyle \lambda\int_0^t  \int_{\R^\ell_+} n_{a,1}(z) \Pi_{w(s)}\left(\R_+,\diff z\right)\diff s{+}\int_0^t \int_{\R{\times}\R^\ell_+} \beta(x) n_{a,2}(z)\Pi_{w(s)}(\diff x,\diff z) \diff s.
\\\displaystyle \frac{\diff w}{\diff t}(t)=M(\omega_p(t),\omega_d(t), w(t))
\end{cases}
\]
\section{Averaging Principle Results}
\label{section:theorems}
We fix $T{>}0$,  throughout the paper the  convergence in distribution of processes is considered on the bounded interval~$[0,T]$.

\subsection{Main Result}
We start by reviewing the assumptions detailed in Section~\ref{defNotSec} on the different parameters of the stochastic model.
\paragraph{\bf Assumptions}
\begin{enumerate}
    \item It is assumed that $\beta$ is a non-negative, continuous function on $\R$ and that $\beta(x){=}0$ for $x{\le}{-}c_\beta{\le}0$.
    Additionally, there exist  a constant  $C_\beta{\geq}0$ such that
    \[
    \beta(x) {\leq} C_{\beta}(1{+}|x|),\quad \forall x{\in}\R;
    \]
    \item $g$ is continuous function on $\R$ and $0{\le}g(x){\le} \max(c_g,x)$ holds for all $x{\in}\R$, for some $c_g{\ge}0$;
    \item All coordinates of the vector $\gamma$ are positive;
    \item There exists a constant $C_k{\geq}0$ such that $0{\le}k_0{\le}C_k$ and functions $k_{i}$, $i{=}1$, $2$, in ${\cal C}_b^1(\R_+^{\ell},\R_+^{\ell})$, are upper-bounded by $C_k{\geq}0$;
    \item There exists a constant $C_n$ such that, for $j{\in}\{0,1,2\}$, $a{\in}\{p,d\}$, $n_{a,j}$ verifies,
        \[
        n_{a,j}(z){\le} C_n(1{+}\|z\|),
        \]
    where, for $z{\in}\R_+^\ell$,  $\|z\|{=}z_1{+}\cdots{+}z_\ell$.
    Moreover, for any $w{\in}K_W$, the discontinuity points of
      \[
      (x,z){\mapsto}(n_{a,0}(z),n_{a,1}(z), \beta(x)n_{a,2}(z))
      \]
      for $a{\in}\{p,d\}$, are negligible for the probability distribution $\Pi_w$ of Section~\ref{section:invariant}.\label{itm:n}
    \item $M$ can be decomposed as, $$M(\omega_p,\omega_d,w){=}M_p(\omega_p,w){-}M_d(\omega_d,w) - \delta w,$$ where $M_{a}(\omega_a,w)$ is non-negative continuous function, non-decreasing on the first coordinate for a fixed $w{\in}K_w$, and,
    \[
       M_a(\omega_a,w)\leq C_M(1{+}\omega_a),
    \]
    for all $w{\in}K_W$, for $a{\in}\{p,d\}$. \label{itm:M}
\end{enumerate}
The main result of the paper is the following theorem.
\begin{theorem}[Asymptotic Time Evolution of Plastic Synapticity]\label{theorem:homog}
Under the conditions of Section~\ref{defNotSec} and for initial conditions satisfying Relation~\eqref{InCond}, there exists $S_0{\in}(0,{+}\infty]$, such that  the family of processes $(\Omega_{\eps,p}(t),\Omega_{\eps,d}(t),W_{\eps}(t),t{<}S_0)$, $\eps{\in}(0,1)$, of the system of Section~\ref{ScaledSec}, is tight for the convergence in distribution.
As $\eps$ goes to $0$, any  limiting point
$(\omega_p(t),\omega_d(t),w(t),t{<}S_0)$, satisfies the ODEs,  for $a{\in}\{p,d\}$,
\begin{equation}\label{ODEH}
\begin{cases}
\displaystyle  \frac{\diff \omega_a}{\diff t}(t)\hspace{-3mm}&\displaystyle{=}{-}\alpha \omega_a(t)
    {+}\hspace{-1mm}\int_{\R{\times}\R_+^\ell}\hspace{-2mm} \left(\rule{0mm}{4mm}n_{a,0}(z){+}\lambda n_{a,1}(z) {+} \beta(x)n_{a,2}(z)\right)\Pi_{w(t)}(\diff x,\diff z),\\[8pt]
\displaystyle \frac{\diff w}{\diff t}(t)\hspace{-3mm}&{=}M(\omega_p(t),\omega_d(t),w(t)),
  \end{cases}
\end{equation}
where, for $w{\in}K_W$, $\Pi_w$ is the unique invariant distribution $\Pi_w$ on $\R{\times}\R_+^\ell$ of the Markovian operator $B_w^F$.
If $K_W$ is bounded, then $S_0{=}{+}\infty$ almost-surely.
\end{theorem}

\paragraph{\sc Convergence in Distribution}
As already mentioned, most of the efforts in this paper are devoted to the proof of  the tightness property of $(\Omega_{\eps,p}(t),\Omega_{\eps,d}(t),W_\eps(t))$.  We note that our result identifies the limiting points, but it does not state any weak convergence results for the scaled processes.
Regularity properties are actually required on $(\Pi_w)$ to have such results. For example, it would be sufficient to have that the mapping,
\[
\Psi_a: w \mapsto \int_{\R{\times}\R_+^\ell} \left(\rule{0mm}{3mm} n_{a,0}(z){+}\lambda n_{a,1}(z){+}\beta(x) n_{a,2}(z)\right)\Pi_{w}(\diff x,\diff z),
\]
locally Lipschitz for $w$, for $a{\in}\{p,d\}$, so that Relation~\eqref{ODEH} has a unique solution. 

Due to the generality of our model,  we did not try to state a set of conditions that can ensure the desired regularity properties  of the corresponding $\Pi_w$. Uniqueness results are obtained in Sections~5 and~6 of~\citet{robert_mathematical_2020} for several important cases. The same properties for the simple model are worked out in Section~\ref{sec:simplemodel}. However at this stage, a case by case analysis seems mandatory.

\bigskip
\paragraph{\sc A Blow-up phenomenon}
As it can be seen, when $S_0{<}{+}\infty$ the convergence is only proved on a bounded time interval.
In the proof, the variable $S_0$ results from the domain definition of the solution to a deterministic differential equation.
This is not an artifact of our methods, see Proposition~\ref{ODESimpAvProp} in Section~\ref{sec:simplemodel} for an example.

\subsection{Steps of the Proof}
The proof of the theorem is organized as follows. See also Figure~\ref{fig:proof} of Appendix.
\begin{enumerate}
\item Section~\ref{section:coupling}.
A stochastic upper-bound $\Dom{U}$ of the original process is introduced and a coupling argument is used to control $(W_\eps(t))$.
This is an important ingredient in the proof of tightness results for $(\Omega_{\eps,p}(t),\Omega_{\eps,d}(t),W_{\eps}(t))$.
\item Section~\ref{OccSec}.
Under the temporary assumption that the process $(\Dom{W}(t))$ is bounded by $K$, we establish tightness results for the truncated process $\Dom{U}^K$, when $\eps$ goes to $0$, of variables associated to fast processes $(\Dom{X}^K_\eps(t),\Dom{Z}^K_\eps(t))$ of the type
\[
\left(\int_{[0,T]} G\left(s,\Dom{X}^K_\eps(s),\Dom{Z}^K_\eps(s)\right)\diff s\right)
\]
where $G$ is a continuous Borelian function with a linear  growth with respect to the coordinates $x{\in}\R$ and $z{\in}\R_+^\ell$.
% It is not enough to consider bounded functions $G$ and apply weak compactness of a family of random measures, see~\citet{karatzas_averaging_1992} for example. However, proving tightness for linear $G$ was one of the non-standard difficulties that we had to face.
An averaging principle is shown for this truncated process.
\item In Theorem~\ref{AsymLinProp} of Section~\ref{section:proof}, using monotonicity arguments, we are able to obtain  a {\em deterministic,  analytical bound}, uniform in $K$,  for the  limiting points of the truncated process.
From there, we prove an averaging principle for the dominating  process $\Dom{U}$ (without truncation) in Proposition~\ref{prop:sapdom} where the explicit form of the ODE verified by the limiting points is known.
As a direct consequence, we are able to prove that this limit is unique and that the scaled dominating process converges to the solution.
Using the fact that the process $W(t)$ is bounded by $\Dom{W}(t)$ and the previous convergence, we establish the desired results for the process $U_{\eps}(t)$ of Theorem~\ref{theorem:homog}.
\end{enumerate}

\subsection{Technical Results on Shot-Noise Processes}
The processes $(X(t))$ and $(Z(t))$ are closely related to shot-noise processes and their generalizations. See for example~\citet{Schottky}, \citet{Rice} and~\citet{gilbert_amplitude_1960} for an introduction. We give a quick overview of their use in our proofs. 
In Appendix~\ref{ShotSec}, the results below and several technical lemmas for these processes are detailed and proved.

The following lemma gives an elementary representation result for general shot-noise process associated to a positive Radon measure.
See Lemma~1 of~\citet{robert_mathematical_2020}.
\begin{lemma}\label{lemma:expofilter}
 If $\mu$ is a positive Radon measure on $\R_+$ and $\gamma{>}0$,  the unique \cadlag solution of the ODE
  \[
  \diff Z(t)={-}\gamma Z(t) {+}\mu(\diff t),
  \]
  with initial point $z_0{\in}\R_+$ is given by
\begin{equation}\label{eqSN}
  Z(t)=z_0e^{-\gamma t}+\int_{(0,t]}e^{-\gamma(t{-}s)}\mu(\diff s).
\end{equation}
\end{lemma}

\noindent
In view of SDEs~\eqref{eq:scaledsimple} it is natural to introduce a scaled version of these processes. 
\begin{definition}[Scaled Shot-Noise Process]\label{SNPdef}
For  $\eps{>}0$, we define the shot-noise process $(S^x_{\eps}(t))$, solution of the SDE
\[
\diff S^x_{\eps}(t)={-}S^x_{\eps}(t)\diff t/\eps+{\cal N}_{\lambda/\eps}(\diff t),
\]
where the initial point is $x{\ge}0$.
\end{definition}

\begin{proposition}\label{lemma:Slambda}
 For $\xi{\in}\R$ and $x{\geq}0$,  the convergence in distribution of the processes
  \[
\lim_{\eps\searrow 0} \left(\int_0^t e^{\xi S^x_{\eps}(u)}\diff u\right)=\left(\E\left[e^{\xi S(\infty)}\right]t\right)
  \]
 holds, and
\begin{equation}\label{eqBouExp}
\sup_{\substack{0{<}\eps{<}1\\0{\leq}t{\leq}T}} \E\left[e^{\xi S_{\eps}(t)}\right] {<}{+}\infty.
\end{equation}
\end{proposition}
\begin{proof}
See Section~\ref{ScSN} of Appendix.
\end{proof}
We now introduce  another shot-noise process $(R_{\eps}(t))$ associated to the point process $\mathcal{N}_{I/\eps,S_{\eps}}$ defined by Relation~\eqref{Nbeta} where $I(x){=}x$, $x{\in}\R_+$. It is in fact a shot-noise process whose intensity function is $(S_{\eps}(t))$,
\begin{equation}\label{Reps}
\diff R_{\eps}(t)  =    {-}\gamma R_{\eps}(t)\diff t/\eps{+}\mathcal{N}_{I/\eps,S_{\eps}}(\diff t),
\end{equation}
with the initial condition $R_\eps(0){=}0$.

It turns out that tightness properties of three families of linear functionals of such processes
\[
\left(\int_0^t R_{\eps}(s)\,\diff s\right),
\left(\int_0^t R_{\eps}(s)\eps{\cal N}_{\lambda/\eps}(\diff s)\right),
\left(\int_0^t R_{\eps}(s)\eps{\cal N}_{I/\eps,S_{\eps}}(\diff s)\right),
\]
are central to establish Theorem~\ref{theorem:homog}.
The motivation comes from the three terms in the  expression of $(\Omega_{\eps,a}(t))$, $a{\in}\{p,d\}$ of Relation~\eqref{eq:markOm} and the fact that, with the condition of Relation~\eqref{Condn}, for $j{\in}\{0,1,2\}$, $n_j(z){\le}C_n^0{+}C_nz$ for $z{\in}\R_+$.

The necessary results are stated in Proposition~\ref{Fam1} which is proved in Section~\ref{TiSecOM} of Appendix,  and then used in Section~\ref{section:proof}.
\begin{proposition}\label{Fam1}
For $H_\eps{\in}\{S_{\eps},R_\eps\}$, the families of processes
\[
\left(\int_0^t H_{\eps}(u)\,\diff u\right), \left(\int_0^t H_{\eps}(u)^2\,\diff u\right)  \text{ and } \left(\int_0^t R_{\eps}(u)S_{\eps}(u)\,\diff u\right),\;\eps{\in}(0,1),
\]
are tight for the convergence in distribution.
\end{proposition}

\section{A Coupling Property}\label{section:coupling}
In  this section  a process
\[
(\Dom{U}(t)){=}(\Dom{X}(t),\Dom{Z}(t), \Dom{\Omega}(t),\Dom{W}(t))
\]
in ${\cal D\left( [0,T], \R_+^4 \right)}$ is introduced.
It has similarities with the process $(U(t))$ of Definition~\eqref{DefPar} but fewer coordinates and simpler parameters.
More importantly, all its coordinates are non-negative. 
We first prove, via a coupling,  that the sample paths of the processes $(U(t))$ and $(\Dom{U}(t))$ can be compared in a sense to be made precise.
Secondly, we derive several technical estimates for $(\Dom{U}(t))$ which are important to prove the tightness of the scaled processes $((\Dom{\Omega}_\eps(t),\Dom{W}_\eps(t)))$ defined in Section~\ref{ScaledSec}.

The process $(\Dom{U}(t))$ is the solution of the SDEs
\begin{equation}\label{eq:markovmaj}
\begin{cases}
    \diff \Dom{X}(t) &\displaystyle = {-}\Dom{X}(t)\diff t+\Dom{W}(t)\mathcal{N}_{\lambda}(\diff t),\\
    \diff  \Dom{Z}(t) &\displaystyle =    \left(-\underline{\gamma} \Dom{Z}(t) +C_k\right)\diff t + C_k\mathcal{N}_{\lambda}(\diff t){+}C_k\mathcal{N}_{\Dom{\beta},\Dom{X}}(\diff t),\\
    \diff \Dom{\Omega}(t)&\displaystyle{=} {-}\alpha\Dom{\Omega}(t)\diff t{+}C_n\left(1{+}\ell\Dom{Z}(t)\right)\diff t{+}
          C_n\left(1{+}\ell\Dom{Z}(t{-})\right)\mathcal{N}_{\lambda}(\diff t)
          \\&\hspace{1cm}+C_n\left(1{+}\ell\Dom{Z}(t{-})\right)\mathcal{N}_{\Dom{\beta},\Dom{X}}(\diff s)\\
    \diff \Dom{W}(t) &\displaystyle = C_M\left(1{+}\Dom{\Omega}(t)\right)\diff t,
\end{cases}
\end{equation}
with $\Dom{\beta}(x){=}C_\beta(1{+}x)$ and with initial condition  $\Dom{U}(0)$ given by
\[
(\Dom{x}_0,\Dom{z}_0, \Dom{\omega}_0,\Dom{W}_0){=}(\max(x_0, 0),\max_{i{\in}\{1,\ldots,\ell\}}\{z_{0,i}\}, \max_{a{\in}\{a,p\}}\{\omega_{0,a}\} , |w_0|).
\]
$C_{\beta}$, $c_{\beta}$, $c_g$, $C_n$, $C_k$ and $C_M$ are non-negative constants associated to the conditions of Section~\ref{ModSec}, and $\underline{\gamma}{\steq{def}}\min(\gamma_{i}:i{=}1,\ldots,\ell)$.

Throughout this section, for $t{\ge}0$, if $(U(t))$ is a solution of Relations~\eqref{eq:markov}, the inequality $U(t){\le}\Dom{U}(t)$ will stand for the four relations,
$X(t){\le}\Dom{X}(t)$,
\[
\max_{i{\in}\{1,\ldots,\ell\}} \left\{ Z_{i}(t) \right\} {\le} \Dom{Z}(t),  \max_{a{\in}\{p,d\}} \left\{\Omega_a(t)\right\} {\le}\Dom{\Omega}(t), \text{ and } |W(t)| {\le} \Dom{W}(t).
\]

\subsection{A Coupling Property}
We start by proving a monotonicity  property of the behavior of both systems ``between'' jumps.

Define $u(t)=(x(t),z(t),\omega_p(t),\omega_d(t), w(t))$ that follows,
\begin{equation*}%\label{DetPla}
\begin{cases}
    \diff x(t) &\displaystyle = {-}x(t)\diff t,\\
    \diff z_{i}(t) &\displaystyle =  \left({-}\gamma_{i} z_{i}(t){+}k_0\right) \diff t,\quad i{\in}\{1,\ldots,\ell\},\\
    \diff \omega_{a}(t)&\displaystyle = {-}\alpha\omega_{a}(t)\diff t{+}n_{0,a}(z(t))\diff t,\quad a{\in}\{p,d\},\\
    \diff w(t) &\displaystyle = M\left(\omega_p(t),\omega_d(t), w(t)\right)\diff t.
\end{cases}
\end{equation*}
and $\Dom{u}(t)=(\Dom{x}(t),\Dom{z}(t),\Dom{\omega}(t),\Dom{w}(t))$ with,
\begin{equation*}%\label{DetPlaM}
\begin{cases}
        \diff \Dom{x}(t) &\displaystyle = {-}\Dom{x}(t)\diff t,\\
        \diff \Dom{z}(t) &\displaystyle =    \left({-}\underline{\gamma} \Dom{z}(t) {+}C_k\right) \diff t,\\
        \diff \Dom{\omega}(t)&\displaystyle = {-}\alpha\Dom{\omega}(t)\diff t{+}C_n\left(1+\ell\Dom{z}(t)\right)\diff t,\\
        \diff \Dom{w}(t) &\displaystyle = C_M\left(1+\Dom{\omega}(t)\right)\diff t.
    \end{cases}
\end{equation*}

\begin{lemma}\label{LemCoup}
Under the conditions of Sections~\ref{SecZ}, \ref{OmCond}, and~\ref{MCond} and if the initial conditions are such that $u(0){\le}\Dom{u}(0)$, then for all $t{\geq}0$, $u(t){\le}\Dom{u}(t)$.
\end{lemma}
\begin{proof}
 The result is clear for the function $(x(t))$ and also for the functions $(z_{i}(t))$.
For $(\omega_a(t))$, with $a{\in}\{p,d\}$, we have
  \[
  \diff \left(\Dom{\omega}(t){-}\omega_{a}(t)\right)  = {-}\alpha\left(\Dom{\omega}(t){-}\omega_{a}(t)\right))\diff t{+}\left(C_n(1{+}\ell\Dom{z}(t)){-}n_{0,a}(z(t))\right)\diff t,
  \]
by  Condition~\eqref{Condn}, we obtain
\[
C_n(1{+}\ell\Dom{z}(t)){-}n_{0,a}(z(t))\ge C_n(1{+}\|z(t)\|){-}n_{0,a}(z(t))\ge 0.
\]
Lemma~\ref{lemma:expofilter} gives the relation
  \[
  e^{\alpha t}\left(\Dom{\omega}(t){-}\omega_{a}(t)\right){=}\left(\Dom{\omega}(0){-}\omega_{a}(0)\right)
  {+}\int_0^t   e^{-\alpha s}C_n\left(1{+}\ell\Dom{z}(s){-}n_{0,a}(z(s))\right)\diff s {\ge}0.
  \]
  Finally, again with Lemma~\ref{lemma:expofilter}, Condition~\ref{MCond} and the last inequality, we have, for $t{\ge}0$,
   \[
   \diff (\Dom{w}(t) - w(t)) = -\delta (\Dom{w}(s)-w(s))\diff t + \left(C_M(1+\Dom{\omega}(s)) - M_p(\omega_p(s), w(s))\right)\diff t\\
   \]
   This leads to,
   \[
       w(t) \leq \Dom{w}(t).
   \]
   In the same way, we can prove that,
    \[
       w(t) \geq -\Dom{w}(t).
   \]
The lemma is proved. 
\end{proof}

\begin{proposition}[Coupling]\label{coupprop}
Under the conditions of Section~\ref{defNotSec}, there exists a coupling of $(\Dom{U}(t))$ and $(U(t))$ such that, almost surely,   for all $t{>}0$, $(U(t)){\le}(\Dom{U}(t))$, in particular
    \[
     |W(t)| \leq  \Dom{W}(t), \quad \forall t{\ge}0.
    \]

\end{proposition}
\begin{proof}
  All we have to prove is that if $U(0){\le}\Dom{U}(0)$ and if $\tau$ is the first jump of either $(U(t))$ or $(\Dom{U}(t))$, then $U(t){\le}\Dom{U}(t)$ for $t{\le}\tau$.
  Our statement is then easily proved by induction on the sequence of jumps of both processes.

  Since $(\Dom{U}(t))$ and $(U(t))$ are governed by the deterministic ODEs of Lemma~\ref{LemCoup}, the relation $U(t){\le}\Dom{U}(t)$ holds for $0{\le}t{<}\tau$.
  The processes $(W(t))$ and $(\Dom{W}(t))$ being continuous, $|W(\tau)|{=}|W(\tau{-})|{\le}\Dom{W}(\tau{-}){=}\Dom{W}(\tau)$.
  
    The instant of jump $\tau$ is the minimum of $\tau_1$, $\tau_2$ and $\Dom{\tau}_2$, with
    \[
        \begin{cases}
            \tau_1=\inf\{t{>}0: {\cal N}_\lambda((0,t]){\ne}0\},\\
            \displaystyle  \tau_2=\inf\left\{t{>}0: {\cal N}_{\beta,X}((0,t]){=}\int_{(0,t]}{\cal P}_2\left(\rule{0mm}{4mm}\left(\rule{0mm}{3mm}0,\beta(X(s{-})\right],\diff s\right){\ne}0\right\}, \\
            \displaystyle  \Dom{\tau}_2=\inf\left\{t{>}0: {\cal N}_{\Dom{\beta},\Dom{X}}((0,t]){=}\int_{(0,t]}{\cal P}_2\left(\rule{0mm}{4mm}\left(\rule{0mm}{3mm}0,\Dom{\beta}(\Dom{X}(s{-}))\right],\diff s\right){\ne}0\right\}.
        \end{cases}
    \]
    Since, for $x{\le}\Dom{x}$, $\beta(x){\le}\Dom{\beta}(x){\le}\Dom{\beta}(\Dom{x})$ is a non-decreasing function and that $X{\le}\Dom{X}$ holds until the first jump, the inequality $\Dom{\tau}_2{\le}\tau_2$ holds  almost surely.

 If $\tau_1{<}\Dom{\tau}_2$, then
 \[
    X(\tau){=}X(\tau{-}){+}W(\tau{-}){\le}\Dom{X}(\tau{-}){+}\Dom{W}(\tau{-}){=}\Dom{X}(\tau).
 \]

 For $i{\in}\{1,\ldots,\ell\}$,
    \[
    Z_{i}(\tau){=}Z_{i}(\tau{-}){+}k_{i}(Z(\tau{-})) \le\Dom{Z}(\tau{-}){+}C_k{=} \Dom{Z}(\tau)
    \]
    and
    \begin{multline*}
      \Omega_{a}(\tau)=\Omega_{a}(\tau{-}){+}n_{a,1}(Z(\tau-))\leq
      \Omega_{a}(\tau{-}){+}C_n\left(1{+}\|Z(\tau)\|\right)\\
      \leq  \Dom{\Omega}(\tau{-}){+}C_n\left( 1 {+} \ell\Dom{Z}(\tau)\right) = \Dom{\Omega}(\tau).
    \end{multline*}
    Thus we have $U(\tau){\le}\Dom{U}(\tau)$.
    The same arguments work in a similar way when $\tau_2{=}\Dom{\tau}_2{<}\tau_1$.

    In this case, we have
    \[
    X(\tau){=}X(\tau{-}){-}g(X(\tau{-})) \leq \Dom{X}(\tau{-}){=}\Dom{X}(\tau).
    \]
    The last case $\Dom{\tau}_2{<}\min(\tau_2,\tau_1)$ is not more difficult, since the components of $(U(t))$ do not experience jumps and those of $(\Dom{U}(t))$ have positive jumps due to ${\cal N}_{\beta,\Dom{X}}$.
    The proposition is proved.
\end{proof}

The process  $(\Dom{U}_\eps(t))$ is defined by the SDEs,
\begin{equation}\label{SDELin}
\begin{cases}
    \diff \Dom{X}_\eps(t) &\displaystyle =  {-}\Dom{X}_\eps(t)\diff t/\eps+\Dom{W}_\eps(t)\mathcal{N}_{\lambda/\eps}(\diff t),\\
    \diff \Dom{Z}_{\eps}(t) &\displaystyle =     \left({-}\gamma \Dom{Z}_{\eps}(t) + C_k\right)\diff t/\eps {+}C_k \mathcal{N}_{\lambda/\eps}(\diff t){+}C_k\mathcal{N}_{\Dom{\beta}/\eps,\Dom{X}_\eps}(\diff t),\\
    \diff \Dom{\Omega}_{\eps}(t)&\displaystyle= {-}\alpha\Dom{\Omega}_{\eps}(t)\diff t{+}C_n\left(1{+}\ell \Dom{Z}_{\eps}(t)\right)\diff t
  \\&\hspace{1cm}+C_n\left(1{+}\ell \Dom{Z}_{\eps}(t{-})\right)\left(\eps \mathcal{N}_{\lambda/\eps}(\diff t){+}\eps \mathcal{N}_{\Dom{\beta}/\eps,\Dom{X}_\eps}(\diff t)\right)\\
    \diff \Dom{W}_\eps(t) &\displaystyle =  C_M\left(1+\Dom{\Omega}_{\eps}(t)\right)\diff t,
\end{cases}
\end{equation}
and with $\Dom{U}_\eps(0){=}\Dom{U}(0){=}(\Dom{x}_0,\Dom{z}_0,\Dom{\omega}_0,\Dom{w}_0)$.

The corresponding infinitesimal generator $\Dom{B}^F_w$ defined by Relation~\eqref{BFGen} is in this case, for $v{=}(x,z)$,
\begin{multline}\label{BFGen2}
\Dom{B}^F_w(f)(v)\steq{def}{-}x\frac{\partial f}{\partial x}(v){+}({-}\gamma z{+}C_k)\frac{\partial f}{\partial z}(v)
\\+\lambda \left(\rule{0mm}{4mm}f(v{+}we_1{+}C_k e_2){-}f(v)\right)
+C_{\beta}(x{+}1)\left(\rule{0mm}{4mm}f(v{+}C_k e_2){-}f(v)\right),
\end{multline}
where $e_1{=}(1,0)$ and $e_2{=}(0,1)$.
\section{Asymptotic Results for the Truncated Process}\label{OccSec}
In this section we study the scaling properties of fast processes of $(\Dom{U}(t))$ defined by Relation~\eqref{eq:markovmaj}.
In this section, we fix $K{>}0$ and consider an analogue process for which the impact of the synaptic weight is truncated at $K$. In Section~\ref{OccSec:3} an averaging principle will be established for this process as a first step in the proof of the main result of the paper.

\subsection{Definition of the Truncated Process}
\label{OccSec:1}

We define  $(\Dom{U}^{K}(t))$ as  the solution of the SDEs,
\begin{equation}\label{SDELinTrunc}
\begin{cases}
    \diff \Dom{X}^{K}_\eps(t) &\displaystyle =  {-}\Dom{X}^{K}_\eps(t)\diff t/\eps+K{\wedge}\Dom{W}^{K}_\eps(t)\mathcal{N}_{\lambda/\eps}(\diff t),\\
    \diff \Dom{Z}^{K}_{\eps}(t) &\displaystyle =     \left({-}\gamma \Dom{Z}^{K}_{\eps}(t) +C_k\right)\diff t/\eps {+}C_k\mathcal{N}_{\lambda/\eps}(\diff t){+}C_k\mathcal{N}_{\Dom{\beta}/\eps,\Dom{X}^{K}_\eps}(\diff t),\\
    \diff \Dom{\Omega}^{K}_{\eps}(t)&\displaystyle= {-}\alpha\Dom{\Omega}^{K}_{\eps}(t)\diff t{+}C_n\left(1{+}\ell \Dom{Z}^{K}_{\eps}(t)\right)\diff t
  \\&\hspace{1cm}+C_n\left(1{+}\ell\Dom{Z}^{K}_{\eps}(t{-})\right)\left(\eps \mathcal{N}_{\lambda/\eps}(\diff t){+}\eps \mathcal{N}_{\Dom{\beta}/\eps,\Dom{X}^{K}_\eps}(\diff t)\right)\\
    \diff \Dom{W}^{K}_\eps(t) &\displaystyle =  C_M\left(1+\Dom{\Omega}^{K}_{\eps}(t)\right)\diff t,
\end{cases}
\end{equation}
and with $\Dom{U}^{K}_\eps(0){=}\Dom{U}(0){=}(\Dom{x}_0,\Dom{z}_0,\Dom{\omega}_0,\Dom{w}_0)$.

We begin with a lemma giving a stochastic upper bound of $(\Dom{X}^K_{\eps}(t))$ in terms of a standard shot-noise process.

\begin{lemma}\label{lemma:majX}
 There exists a constant $C_X{>}0$ independent of $\eps$  such that the  relation
\begin{equation}\label{eqpp2}
\Dom{X}^K_\eps(t)\leq C_X{+}KS_{\eps}(t)
\end{equation}
holds for  $t{\geq}0$, where $(S_\eps(t))$ is the shot-noise process of Definition~\ref{SNPdef}.

For any $\eta{>}0$, there exists a compact subset ${\cal K}$ of $\R_+^2$  such that,
    \[
 \sup_{\substack{0{<}\eps{<}1\\0{\leq}t{\leq}T}}\P\left(\left(\Dom{X}^K_{\eps}(t),\Dom{Z}^K_{\eps}(t)\right){\notin}{\cal K}  \right)\leq \eta.
    \]
\end{lemma}
\begin{proof}
With Relation~\eqref{eqSN}, we have, for $t{\ge}0$,
\[
\Dom{X}^K_\eps(t) = x_0e^{-t}{+}\int_0^t e^{{-}(t-u)/\eps}K{\wedge}W_\eps(u{-})\mathcal{N}_{\lambda/\eps}(\diff u).
\]
Which gives, for $s{\le}t{\le}T$,
\[
 \Dom{X}^K_\eps(t)\leq x_0e^{-t/\eps}{+}K\int_0^t e^{-(t-u)/\eps}{\cal N}_{\lambda/\eps}(\diff u)\le x_0{+}KS_{\eps}(t),
\]
and therefore
\[
\E\left[\Dom{X}^K_\eps(t)\right]\leq C_X{+}K\E\left[S_{\eps}(t)\right]\leq C_X{+}\lambda K T.
\]
Relation~\eqref{eqSN} gives the inequality
 \[
 \E\left[\Dom{Z}^K_{\eps}(t)\right] \leq z_0{+} C_k\int_0^t \exp(-\gamma (t-s))\left(1 {+} \lambda {+} C_\beta(1{+}\E\left[\Dom{X}^K_\eps(s)\right])\right)\diff s
    \]
    which leads to,
    \[
      \sup_{\substack{0{<}\eps{<}1\\0{\leq}t{\leq}T}}  \E\left[ \Dom{Z}^K_{\eps}(t)\right] < +\infty.
    \]
We conclude by using Markov's inequality.
\end{proof}

\subsection{Tightness of the Truncated Process}
\label{OccSec:2}
The next important lemma is used to prove tightness properties of the processes $(\Omega_\eps(t))$.
\begin{lemma}[Tightness of Linear Functionals of the Fast Processes]\label{ITight}
The family of processes
\[
\left(\int_0^t \Dom{X}_\eps^K(u)\diff u\right),
\left(\int_0^t \Dom{Z}_\eps^K(u)\diff u\right),
\left(\int_0^t \Dom{X}_\eps^K(u)\Dom{Z}_\eps^K(u)\diff u\right), \;\eps{\in}(0,1),
\]
are tight for the convergence in distribution. The processes
\begin{align*}
(\Dom{M}^K_{\eps,1}(t))&{\steq{def}}\left(\int_0^t \Dom{Z}^K_\eps(u{-})\left[\eps{\cal N}_{\lambda/\eps}(\diff u){-} \lambda\diff u\right]\right)\\
(\Dom{M}^K_{\eps,2}(t))&{\steq{def}}\left(\int_0^t \Dom{Z}^K_\eps(u{-})\left[\rule{0mm}{4mm}\eps{\cal N}_{\Dom{\beta}/\eps,\Dom{X}^K_{\eps}}(\diff u){-} \Dom{\beta}\left(\Dom{X}^K_{\eps}(u)\right)\diff u\right]\right)
\end{align*}
converge in distribution to $0$ as $\eps$ goes to $0$. 
\end{lemma}
\begin{proof}
Relation~\eqref{eqpp2} gives for $0{\le}s{\le}t$,
  \[
  \int_s^t \Dom{X}_\eps^K(u)\diff u\leq C_X(t{-}s)+K\int_s^t S_\eps(u)\diff u,
  \]
The tightness of the three processes results from this relation and Proposition~\ref{Fam1}.

Indeed, Relation~\eqref{eqSN} shows that, for $t{\ge}0$,
\begin{align*}
  \Dom{Z}_\eps^K(t){-}z_0&=C_k\int_0^t e^{-\gamma(t-s)/\eps} \diff s+C_k\int_0^t e^{-\gamma(t-s)/\eps}{\cal N}_{\lambda/\eps}(\diff s)
  \\&\hspace{1cm}+C_k\int_0^t e^{-\gamma(t-s)/\eps}{\cal P}_2\left(\left(0,C_\beta\frac{1{+}\Dom{X}^K_\eps(s)}{\eps}\right],\diff s\right)\\
&\le \frac{C_k}{\gamma} + C_k\int_0^t e^{-\gamma(t-s)/\eps}{\cal N}_{\lambda/\eps}(\diff s)
  \\&\hspace{8mm}+ C_k\int_0^t e^{-\gamma(t-s)/\eps}{\cal P}_2\left(\left(0,C_\beta\frac{1{+}C_X}{\eps}\right],\diff s\right)\\
&\hspace{8mm}    + C_k\int_0^t e^{-\gamma(t-s)/\eps}{\cal P}_2\left(\left(C_\beta\frac{1{+}C_X}{\eps},C_\beta\frac{1{+}C_X}{\eps}{+}C_\beta K\frac{S_\eps(s)}{\eps}\right],\diff s\right).
\end{align*}
The first two terms of the right-hand side of last relation are, up to the constants $\gamma$ instead of $1$ and $C_\beta(1{+}C_X)$ instead of $\lambda$,  equal to $S_\eps(t)$.
Similarly, up to the constant $C_\beta K$ of $S_{\eps}(t)$ instead of $1$, the last term is equal to $R_\eps(t)$.

The two processes $(\Dom{M}^K_{\eps,i}(t))$, $i{=}\{1,2\}$ are martingales with previsible increasing processes
  \[
  \left(\eps\lambda \int_0^t \Dom{Z}^K_\eps(u)^2\diff u\right)\text{ and } 
  \left(\eps\int_0^t \Dom{Z}^K_\eps(u)^{2}C_{\beta}\left(1{+}\Dom{X}^K_{\eps}(u)\right)\diff u\right).
  \]
For $t{\ge}0$, we have
\[
    \E\left[\Dom{M}^K_{\eps,1}(t)^2\right] \leq \int_0^t \E\left[\Dom{Z}^K_\eps(u)^2\right]\diff u,
\]
and,
\begin{multline*}
  \E\left[\Dom{M}^K_{\eps,2}(t)^2\right]{=}\eps C_\beta\left(\E\left[\Dom{M}^K_{\eps,1}(t)^2\right]+ \int_0^t \E\left[\Dom{Z}^K_\eps(u)^2\Dom{X}^K_{\eps}(u)\right]\diff u\right)\\
\le
  \eps C_\beta\int_0^t\left(\E\left[\Dom{Z}^K_\eps(u)^2\right] + \sqrt{\E\left[\Dom{Z}^K_\eps(u)^4\right]}\sqrt{\E\left[\Dom{X}^K_{\eps}(u)^2\right]}\right)\diff u,
\end{multline*}
 with Cauchy-Schwartz' inequality.

 Using  the upper-bounds for  $(\Dom{X}^K_\eps(t))$ and $(\Dom{Z}^K_\eps(t))$ and Relation~\eqref{eqBouExp} for $S_{\eps}$ and Proposition~\ref{R4Mom} of the Appendix for $R_{\eps}$, we obtain that the quantity $\E\left[\Dom{M}^K_{\eps,2}(t)^2\right]$ converges to $0$ as $\eps$ goes to $0$.
 The last statement of the lemma follows from Doob's inequality.
\end{proof}

We now define the associated occupation measure $\Dom{\nu}^K_{\eps}$ in the same way as in Section~\ref{OccMeas}. Let $G$ be a non-negative Borelian function on $[0,T]{\times}\R^2_+$,  define $\Dom{\nu}^K_\eps$ the non-negative measure on $[0,T]{\times}\R^2_+$ by
\[
\int_{[0,T]{\times}\R^2_+}\hspace{-3mm} G(s,x,z)\Dom{\nu}^K_{\eps}(\diff s,\diff x,\diff z)\steq{def}\int_{[0,T]} G\left(s,\Dom{X}^K_\eps(s),\Dom{Z}^K_\eps(s)\right)\diff s.
\]
\begin{lemma}\label{TOccM}
The family of random Radon measures $\Dom{\nu}^K_{\eps}$, $\eps{\in}(0,1)$, is tight for the convergence in distribution and for any bounded Borelian function on $[0,T]{\times}\R_+^2$, the set of processes
  \[
(I_G(t))\steq{def} \left(\int_0^t G\left(u, \Dom{X}^K_{\eps}(u), \Dom{Z}^K_{\eps}(u)\right)\,\diff u,0{\le}t{\le}T\right)
  \]
is tight for the convergence in distribution.
\end{lemma}
\begin{proof}
For $a{>}0$, $\eps{\in}(0,1)$, and ${\cal K}$  a Borelian subset of $\R_+^2$, we have 
\begin{multline*}
\P\left(\rule{0mm}{4mm}\Dom{\nu}^K_\eps([0,T]{\times}{\cal K}^c){>}a\right)\leq \frac{1}{a}\E\left[\Dom{\nu}^K_\eps([0,T]{\times}{\cal K}^c)\right]
\\ \leq \frac{T}{a}\sup_{\substack{0{<}\eps{<}1\\0{\leq}t{\leq}T}}\P\left(\left(\Dom{X}^K_{\eps}(t),\Dom{Z}^K_{\eps}(t)\right){\notin}{\cal K}\right)
\end{multline*}
Lemma~\ref{lemma:majX} shows the existence of a compact set ${\cal K}{\subset}\R_+^2$ such that the last term of the right-hand side of this inequality can be made arbitrarily small.
Lemma~1.3 of~\citet{karatzas_averaging_1992}  gives that the family of random measures $(\Dom{\nu}^K_{\eps}, 0{<}\eps{<}1)$ is tight.

for the last part of the proposition, we use the criterion of modulus of continuity, see~\citet{billingsley_convergence_1999}.
This is a simple consequence of the inequality, for $0{\le}s, t{\le}T$, $|I_G(t){-}I_G(s)|{\leq}\|G\|_\infty |t{-}s|$.
The lemma is proved.
\end{proof}

\begin{proposition}\label{TightBoundProp}
The family of random variables $(\Dom{\Omega}^{K}_\eps(t),\Dom{W}^{K}_\eps(t), \Dom{\nu}_{\eps}^{K})$, $\eps{\in}(0,1)$, is tight.
\end{proposition}
\begin{proof}
Tightness properties of $(\Dom{\nu}_\eps^{K})$ have been proved in Lemma~\ref{TOccM}.
Relation~\eqref{eqSN} gives the relation, for $t{\ge}0$,
\begin{multline}\label{fgh1}
\Dom{\Omega}^{K}_{\eps}(t){-}\Dom{\omega}_{0}e^{-\alpha t}{=} \int_0^t e^{-\alpha(t-s)}C_n\left(1{+}\ell \Dom{Z}^{K}_{\eps}(s)\right)\left(1{+}\lambda{+}C_{\beta}(1{+}\Dom{X}^{K}_\eps(s))\right)\diff s\\
{+}\int_0^t\hspace{-2mm}e^{-\alpha(t-s)}C_n\hspace{-1mm}\left(1{+}\ell \Dom{Z}^{K}_{\eps}(s{-})\right)\hspace{-1mm}\left[\eps\mathcal{N}_{\lambda/\eps}(\diff s){-}\lambda\diff s{+} \eps{\cal N}_{\Dom{\beta}/\eps,\Dom{X}^{K}_\eps}(\diff s){-}\Dom{\beta}\left(\Dom{X}^K_\eps(s)\right)\diff s \right].
\end{multline}
Lemma~\ref{ITight} shows that the family of processes associated to the first term of the right-hand side of this identity is tight, and that the process of the second term is vanishing in distribution as $\eps$ goes to $0$.
The family of processes $(\Dom{\Omega}^{K}_{\eps}(t))$ is therefore tight and the tightness of $(\Dom{W}^{K}_\eps(t))$ follows from its representation with $(\Dom{\Omega}^{K}_{\eps}(t))$.
The proposition is proved.
\end{proof}

\subsection{Averaging Principle for the  Truncated Process $(\Dom{U}^K(t))$}
\label{OccSec:3}
The goal of this section is to prove the following averaging principles for the truncated process.
We start by stating the two following lemmas that are proved in Appendix~\ref{sec:aproof} and that are essential to the proof of averaging principle.

We fix a sequence $(\eps_n)$ such that $(\Dom{\nu}_{\eps_n}^K)$ is converging in distribution to $\Dom{\nu}^K$.
The first result focus on identifying the limiting linear functional of $\Dom{\nu}^K$.
\begin{lemma}\label{CVX2}
For any continuous bounded Borelian function $G$ and $a$, $b$, $c{\in}\R_+$,  the sequence of processes
\begin{align*}
 & \left(\int_0^t \left(\rule{0mm}{4mm}a\Dom{X}_{\eps_n}^K(s){+}b \Dom{Z}_{\eps_n}^K(s){+}c\Dom{X}_{\eps_n}^K(s)\Dom{Z}_{\eps_n}^K(s)\right)G\left(\Dom{X}_{\eps_n}^K(s),\Dom{Z}_{\eps_n}^K(s)\right)\diff s\right)
  \intertext{ converges  in distribution to }
&\left(\int_0^t\int_{\R^2} (ax{+}bz{+}cxz)G(x,z)\Dom{\nu}^K(\diff s, \diff x,\diff z)\right).
\end{align*}
\end{lemma}

The second lemma shows that $\Dom{\nu}^K$ can be expressed as the product of the invariant measure $\Pi_w$ and the Lebesgue measure, much in the spirit of~\citet{karatzas_averaging_1992}.
\begin{lemma}\label{OccMeasProp}
For any non-negative Borelian function $F$ on $\R_+{\times}\R_+^2$, almost surely,
\[
\int_0^T F(s,x,z)\Dom{\nu}^{K}(\diff s, \diff x,\diff z) =   \int_0^T F(s,x,z)\Pi_{K{\wedge}\Dom{w}^{K}(s)}(\diff x,\diff z)\,\diff s,
\]
where, for $w{\in}\R_+$, $\Pi_{w}$ is the unique invariant distribution of the Markov process associated to the infinitesimal generator $B^F_w$ defined by Relation~\eqref{BFGen2}.
\end{lemma}
With these two lemmas, Proposition~\ref{HomTruncProp} can be established.

\begin{proposition}\label{HomTruncProp}
Any limiting point $(\Dom{\omega}^{K}(t),\Dom{w}^{K}(t))$ of the family of processes  $(\Dom{\Omega}^{K}_{\eps}(t),\Dom{W}^{K}_\eps(t))$, when $\eps$ goes to $0$, verifies, almost surely for all $t{\ge}0$, the ODE
\[
\begin{cases}
\Dom{\omega}^{K}(t)&\displaystyle= \Dom{\omega}_{0}{-}\alpha \int_0^t \Dom{\omega}^{K}(s)\diff s \\
&\displaystyle\hspace{1cm}+\int_0^t \int_{\R^2}C_n(1{+}\ell z)(1{+}\lambda{+}C_{\beta}(1{+}x))\Pi_{\Dom{w}^{K}(s){\wedge} K}(\diff x,\diff z)\,\diff s\\
\Dom{w}^{K}(t)&\displaystyle=\Dom{w}_0+\int_0^t C_M\left(1+\Dom{\omega}^{K}(s)\right)\,\diff s,
\end{cases}
\]
holds, where $\Pi_{w}$ is the unique invariant distribution of the Markov process associated to the infinitesimal generator $B^F_w$ defined by Relation~\eqref{BFGen2}.
\end{proposition}

\begin{proof}
Relation~\eqref{fgh1} gives the identity, for $t{\ge}0$,
 \[
 \Dom{\Omega}^{K}_{\eps}(t)=\Dom{\omega}_0 e^{-\alpha t}{+} e^{-\alpha t}\Dom{M}_\eps^{K}(t){+} e^{-\alpha t} \int_0^t
\hspace{-2mm}e^{\alpha s}C_n(1{+}\ell \Dom{Z}_{\eps}(s))\left(1{+}\lambda{+}C_{\beta}(1{+}\Dom{X}^{K}_\eps(s))\right)\, \diff s,
 \]
with
\[
\Dom{M}_\eps^{K}(t)\steq{def}  \int_0^te^{\alpha s}C_n\left(1{+}\ell \Dom{Z}^{K}_{\eps}(s)\right)\left[\eps\mathcal{N}_{\lambda/\eps}(\diff s){+} \eps{\cal N}_{\Dom{\beta}/\eps,\Dom{X}^{K}_\eps}(\diff s){-}(\lambda{+}\Dom{\beta}(\Dom{X}_\eps^{K}(s)))\diff s\right].
\]
Proposition~\ref{ITight} shows that $(\Dom{M}_\eps^{K}(t))$ is  converging in distribution to $0$ when $\eps$ goes to $0$.
We now use Lemmas~\ref{CVX2} and~\ref{OccMeasProp} and we get that $(\Dom{\omega}^{K}(t),\Dom{w}^{K}(t))$ satisfies the desired relation.
\end{proof}
\section{Proof of an Averaging Principle}
\label{section:proof}
Finally, this section gathers all the results from the previous sections to prove Theorem~\ref{theorem:homog}.
The proof is done in two steps:
\begin{enumerate}
\item  Using  an analytical result, an averaging principle for $(\Dom{U}(t))$ is proved.
\item  The coupling of Section~\ref{section:coupling} is then used to show that a stochastic averaging result also holds in the general case.
\end{enumerate}

\subsection{Averaging Principle for the Coupled Process $(\Dom{U}(t))$}
\label{section:proof-1}
We now turn to an analytical result by considering the dynamical system of Proposition~\ref{HomTruncProp} when $K{=}{+}\infty$ and by showing an existence and uniqueness results which will be crucial in the proof of the general theorem.

For $w{\ge}0$, $\Dom{\Pi}_w$ is the invariant distribution  of the Markov process $(\Dom{X}^w(t),\Dom{Z}^w(t))$ satisfying the SDE
\begin{align*}
\diff \Dom{X}^w(t) &= {-}\Dom{X}^w(t)\diff t+w\mathcal{N}_{\lambda}(\diff t),\\
\diff \Dom{Z}^w(t) &= \left(-\gamma \Dom{Z}^w(t){+}C_k\right)\diff t{+}C_k\mathcal{N}_{\lambda}(\diff t){+}C_k\mathcal{N}_{\Dom{\beta},\Dom{X}^w}(\diff t).
\end{align*}
Its existence is a consequence of Proposition~\ref{InvPropFP}. 
\begin{theorem}\label{AsymLinProp}
Under conditions of Section~\ref{MCond},  there exists $S_0{\in}(0,{+}\infty]$ and  a unique continuous function $(\Dom{\omega}(t),\Dom{w}(t))$ on $[0,S_0)$, solution of the ODE,  for  $0{\le}t{<}S_0$,
    \begin{equation}\label{AsymLinODE}
    \begin{cases}
    \Dom{\omega}(t)&\displaystyle \hspace{-3mm}=\Dom{\omega}_0{-}\alpha \!\!\int_0^t \!\!\Dom{\omega}(s)\diff s{+}\!\!\int_0^t\hspace{-2mm} \int_{\R_+^2}C_n(1{+}\ell z)(1{+}\lambda{+}C_{\beta}(1{+}x))\Dom{\Pi}_{\Dom{w}(s)}(\diff x,\diff z)\diff s,\\
    \Dom{w}(t)&\displaystyle \hspace{-3mm}=\Dom{w}_0+\int_0^t C_M\left(1+\Dom{\omega}_p(s)\right)\,\diff s,
      \end{cases}
    \end{equation}
    with $(\Dom{\omega}_{0},\Dom{w}_0){\in}\R_+^2$

    Any limiting point $(\Dom{\omega}^{K}(t),\Dom{w}^{K}(t)))$ of the family of  processes  $(\Dom{\Omega}^{K}_{\eps}(t),\Dom{W}^{K}_\eps(t))$,  when $\eps$ goes to $0$ is such that, for all $0{\le}t{<}S_0$,
 \[
\Dom{\omega}^{K}(t)\le \Dom{\omega}(t) \text{ and } \Dom{w}^{K}(t)\le \Dom{w}(t).
 \]
\end{theorem}

\begin{proof}
The existence of limiting points of the  processes  $(\Dom{\Omega}^{K}_{\eps}(t),\Dom{W}^{K}_\eps(t))$ is due to Proposition~\ref{HomTruncProp}. 
If $(\Dom{X}^w,\Dom{Z}^w)$ is a random variable with distribution $\Dom{\Pi}_w$, we have
\[
\Dom{X}^w\steq{dist}w\int_0^{+\infty} e^{-s}{\cal N}_\lambda(\diff s),
\]
and, with standard calculations, we obtain the relations
\begin{equation}\label{SNnRS}
\E\left[\Dom{X}^w\right]{=}\lambda w,\qquad
\E\left[\left(\Dom{X}^w\right)^2\right]{=}\left(\lambda^2 {+}\frac{\lambda}{2}\right)w^2,
\end{equation}
and, consequently,
\[
\gamma \E\left[\Dom{Z}^w\right]=C_k\left(1{+}\lambda{+}C_{\beta}\left(1{+}\E\left[\Dom{X}^w\right]\right)\right)=C_k\left(1{+}\lambda{+}C_{\beta}(1{+}\lambda w)\right).
\]
The SDEs for $(\Dom{X}^w(t))$ and $(\Dom{Z}^w(t))$ give
\begin{multline*}
\diff \Dom{X}^w\Dom{Z}^w(t)=\left({-}(\gamma{+}1)\Dom{X}^w(t) \Dom{Z}^w(t){+} C_k\Dom{X}^w(t)\right)\diff t\\{+}(w\Dom{Z}^w(t{-}){+}C_k\Dom{X}^w(t{-}){+}C_k w)\mathcal{N}_{\lambda}(\diff t){+}C_k\Dom{X}^w(t-)\mathcal{N}_{\Dom{\beta},\Dom{X}^w}(\diff t),
\end{multline*}
and thus, at equilibrium, we obtain the relation
\[
\E\left[\Dom{X}^w\Dom{Z}^w\right]=C_k\frac{1}{\gamma{+}1}\left(\lambda w\left(1{+}\E\left[\Dom{Z}^w\right]\right){+}(1{+}\lambda{+}C_{\beta})\E\left[\Dom{X}^w\right]{+}C_{\beta}\E\left[\left(\Dom{X}^w\right)^2\right]\right).
\]
We have therefore that the function
\[
w\mapsto C_n\int_{\R_+^2}(1{+}\ell z)(1{+}\lambda{+}C_{\beta}(1{+}x))\Dom{\Pi}_{w}(\diff x,\diff z)
\]
is a non-decreasing and locally Lipschitz function.
The existence and uniqueness follows from standard results for ODEs.
There exists some $S_0{>}0$, such that, on the time interval $[0,S_0)$, the solution $(\Dom{\omega}(t),\Dom{w}(t))$ of the ODE is the limit of a Picard's scheme $(\Dom{\omega}_{n}(t),\Dom{w}_n(t))$ associated to Relation~\eqref{AsymLinODE} with $$(\Dom{\omega}_{0}(t),\Dom{w}_0(t)){=}(\Dom{\omega}^{K}(t),\Dom{w}^{K}(t)),$$ for all $K{\in}\R_+$.
See Section~3 of Chapter~8 of~\citet{Hirsch} for example.
We now prove by induction that
$(\Dom{\omega}^{K}(t),\Dom{w}^{K}(t)){\le}(\Dom{\omega}_{n}(t),\Dom{w}_n(t))$ holds on  $[0,S_0)$ for all $n{\ge}1$.
If this is true for $n$, then
\begin{align*}
  \Dom{\omega}_{n+1}(t)&=\Dom{\omega}_{0}e^{-\alpha t}{+} \int_0^te^{-\alpha (t-s)}\int_{\R^2}C_n(1{+}\ell z)(1{+}\lambda{+}C_{\beta}(1{+}x))\Pi_{\Dom{w}_n(s)}(\diff x,\diff z)\,\diff s\\
&\ge\Dom{\omega}_{0}e^{-\alpha t}{+} \int_0^t\hspace{-2mm} e^{-\alpha (t-s)}\int_{\R^2}C_n(1{+}\ell z)(1{+}\lambda{+}C_{\beta}(1{+}x))\Pi_{\Dom{w}^{K}(s){\wedge}K}(\diff x,\diff z)\,\diff s\\
&=\Dom{\omega}^{K}(t),
\end{align*}
and the relation $\Dom{w}_{n+1}(t){\ge}\Dom{w}^{K}(t)$ follows directly.
The proof by induction is completed.
We just have to let $n$ go to infinity to obtain the last statement of our proposition.
\end{proof}

\begin{proposition} \label{prop:sapdom}
  Under conditions of Section~\ref{MCond}, for the convergence in distribution,
\[
\lim_{\eps\to 0} ((\Dom{\Omega}_{\eps}(t), \Dom{W}_{\eps}(t)),t{<}S_0) =
((\Dom{\omega}(t),\Dom{w}(t)),t{<}S_0),
\]
where $(\Dom{\Omega}_{\eps}(t), \Dom{W}_{\eps}(t))$ is the process defined by SDEs~\eqref{SDELin} and $((\Dom{\omega}(t),\Dom{w}(t)),t{<}S_0)$ by ODE~\eqref{AsymLinODE}.
\end{proposition}

\begin{proof}
From Proposition~\ref{HomTruncProp}, let $(\Dom{\omega}^{K}(t),\Dom{w}^{K}(t))$ be a limiting point, there exists a sequence $(\eps_n)$ such that  the sequence of processes $(\Dom{\Omega}^{K}_{\eps_n}(t),\Dom{W}^{K}_{\eps_n}(t))$ is converging to a continuous process  $(\Dom{\omega}^{K}(t),\Dom{w}^{K}(t))$.

With the same notations as in Proposition~\ref{AsymLinProp}, for any $T{<}S_0$, by continuity of $(\Dom{\omega}(t),\Dom{w}(t))$ on $[0,T]$, the quantity
\[
K_0{\steq{def}}1{+}\sup_{t{\le}T} \Dom{w}(t)
\]
is finite.
Since $\Dom{w}^K(t){\le}\Dom{w}(t)$ holds for all $t{\ge}0$, the uniqueness result of  Proposition~\ref{AsymLinProp} gives the identity
\[
((\Dom{\omega}^{K}(t),\Dom{w}^{K}(t)),t{\le}T){=}((\Dom{\omega}(t),\Dom{w}(t)),t{\le}T)
\]
for all $K{\ge}K_0$.
Consequently, for any $\eta{>}0$, there exists $n_0{>}0$ such that for $n{\ge}n_0$,
\[
\P\left(\sup_{s\le T} \Dom{W}^{K_0}_{\eps_n} (s)\geq K_0 \right) \le \P\left(\sup_{s\le T} \Dom{W}^{K_0}_{\eps_n} (s)\geq  1{+}\sup_{t{\le}T} \Dom{w}^{K_0}(t) \right)\le \eta,
\]
since  the process $(\Dom{\omega}^{K_0}(t),\Dom{w}^{K_0}(t))$ is upper-bounded, coordinate by coordinate  on the time interval $[0,S_0)$, by $(\omega(t),w(t))$, defined by Relation~\eqref{AsymLinODE}. Note that $S_0$ is independent of the sequence $(\eps_n)$.
Hence,  for  $n{\ge}n_0$, Relation~\eqref{SDELinTrunc}  gives
\begin{multline}\label{equiK}
  \P
  \begin{pmatrix}
   \Dom{X}_{\eps_n}(s){=}\Dom{X}_{\eps_n}^{K_0}(s),  \Dom{Z}_{\eps_n}(s){=}\Dom{Z}_{\eps_n}^{K_0}(s),\\
   \Dom{\Omega}_{\eps_n}(s){=} \Dom{\Omega}^{K_0}_{\eps_n}(s), \Dom{W}_{\eps_n}(s){=}\Dom{W}_{\eps_n}^{K_0}(s),  \forall s{\le}T
  \end{pmatrix} \\
\ge\P\left(\sup_{ s{\le}T } \Dom{W}_{\eps_n}^{K_0}(s){\leq}K_0\right)\ge 1{-}\eta.
\end{multline}

This shows that the sequence of processes $((\Dom{\Omega}_{\eps_n}(t), \Dom{W}_{\eps_n}(t)),t{\le}T)$ is converging in distribution to  $((\Dom{\omega}(t),\Dom{w}(t)),t{\le}T)$. The proposition is proved.
\end{proof}

\subsection{Averaging Principle for the Process $({U}(t))$}
\label{secsec:theoremproof}
We now conclude this section with a sketch of the proof of Theorem~\ref{theorem:homog}. The missing details are not difficult to complete since we have already proved the main difficult results. We fix $T{<}S_0$. 

The  coupling property of Proposition~\ref{coupprop} gives, for the scaled processes, the relation  $U_\eps(t){\le}\Dom{U}_\eps(t)$,  for any $\eta{>}0$, Relation~\eqref{equiK} gives the existence of $K_0$ and $n_0$ such that $n{\ge}n_0$, the inequality
\[
\P\left(\sup_{t{\le}T} \left|W_{\eps_n}(t)\right|{\le} K_0\right)\ge 1{-}\eta
\]
holds.

We get that the results of Lemma~\ref{ITight} hold with $\Dom{X}_{\eps_n}^K$ and $\Dom{Z}_{\eps_n}^K$ replaced by $X_{\eps_n}$ and $\|Z_{\eps_n}\|$, and $\Dom{\beta}$ by $\beta$.
With the same arguments as in the proof of Lemma~\ref{TOccM}, the family of random measures $(\nu_{\eps_n})$ defined by Relation~\eqref{OddMDef} is tight.
Again the analogue of Lemma~\ref{CVX2} also holds.
Finally we have the tightness of the sequence of processes
\[
(\Omega_{{\eps_n},p}(t),\Omega_{{\eps_n},d}(t),W_{\eps_n}(t),\nu_{\eps_n}).
\]
By taking a subsequence, we can assume it is converging in distribution to some process $(\omega_p(t),\omega_d(t),w(t),\nu)$.
The proof of the averaging principle is built on the same architecture.
We just point out the places where the specific conditions of Section~\ref{defNotSec} play a role.
\begin{enumerate}
\item {\em Identification of $\nu$.}\\
If $B^f_w$ is the operator defined by Relation~\eqref{BFGen}, the continuity property of $g(\cdot)$, Condition of Section~\ref{SecPost}, and of $k_{i}$, $i{\in}\{1,2\}$, Condition of Section~\ref{SecZ} give that $(x,z,w){\mapsto}B^F_{w}(f)(x,z)$ is continuous and, with the analogue of Lemma~\ref{CVX2}, we have,
\[
\lim_{n\to+\infty} \int_0^T B^F_{W_{\eps_n}(s)}(f)\left(X_{\eps_n}(s),Z_{\eps_n}(s)\right) \diff s = \int_0^T \int_{\R{\times}\R_+^{\ell}}B^F_{w(s)}(f)(x,z)\nu(\diff s,\diff x,\diff z)
\]
Moreover, if $f{\in}{\cal C}_b^1(\R{\times}\R_+^{\ell})$,
\begin{equation}\label{jka2}
f\left(X_{\eps_n}(s),Z_{\eps_n}(s)\right)=f(x_0,z_0){+}M_{\eps_n}^f(t){+}\frac{1}{\eps}\int_0^t B^F_{W_{\eps_n}(s)}(f)\left(X_{\eps_n}(s),Z_{\eps_n}(s)\right) \diff s.
\end{equation}
as in the proof of Lemma~\ref{OccMeasProp}, we have that $(\eps_n M_{\eps_n}^f(t))$ is converging in distribution to $0$ which leads to the fact that
\[
\int_0^T \int_{\R{\times}\R_+^{\ell}}B^F_{w(s)}(f)(x,z)\nu(\diff s,\diff x,\diff z){=}0,
\]
and, consequently, almost surely, for any bounded Borelian function $G$ on $\R_+{\times}\R{\times}\R_+^{\ell}$,
\[
\int_0^T G(s,x,z)\nu(\diff s,\diff x,\diff z)=\int_0^T\int_{\R{\times}\R_+^{\ell}} G(s,x,z)\Pi_{w(s)}(\diff x,\diff z)\diff s.
\]
This important results from the fact that the invariant distribution $\Pi_w$ exists and is unique, for every $w{\in}K_W$ as proved in Proposition~\ref{InvPropFP} of the Appendix.
\item To establish the first identity of Relation~\eqref{AsymLinODE}, we need the convergence in distribution
\begin{multline*}
  \lim_{n\to+\infty} \left(\int_0^t e^{-\alpha(t-u)}\begin{pmatrix} n_{0}\left(Z_{\eps_n}(u)\right)\\n_{1}\left(Z_{\eps_n}(u)\right)\\\beta\left(X_{\eps_n}(u) \right)n_{2}\left(Z_{\eps_n}(u)\right)
  \end{pmatrix}\diff u \right)=\\
\left(\int_0^t e^{-\alpha(t-u)}\int_{\R{\times}\R_+^\ell}\begin{pmatrix} n_{0}\left(z\right)\\n_{1}\left(z\right)\\\beta\left(x\right)n_{2}\left(z\right)
\end{pmatrix}
\Pi_{w(u)}(\diff x,\diff z)\diff u\right).
\end{multline*}
This is consequence of the fact that $(w(t))$ is almost surely continuous and that, with the conditions of Section~\ref{OmCond}, for any $w{\in}K_W$,
  \[
  (x,z){\mapsto}(n_{0}(z), n_{1}(z),\beta(x) n_2(z)),
  \]
is $\Pi_w$ almost everywhere continuous. 
\end{enumerate}
The theorem is proved. 
\section{The Simple Model}\label{sec:simplemodel}
In this section we consider the simple model defined in Section~\ref{secsec:simplemodel}. 
Recall that the associated SDEs are
\[
\begin{cases}
\diff X(t) \displaystyle = {-}X(t)\diff t+W(t)\mathcal{N}_{\lambda}(\diff t),\\
\diff Z(t) \displaystyle =   {-}\gamma Z(t) \diff t+B_1\mathcal{N}_{\lambda}(\diff t)+B_2\mathcal{N}_{\beta,X}(\diff t),\\
\diff W(t) \displaystyle = Z(t{-})\mathcal{N}_{\beta,X}(\diff t),
\end{cases}
\]
with $\gamma{>}0$, $B_1$, $B_2{\in}\R_+$, and $\beta$ is assumed to be a Lipschitz function on $\R_+$.

This is not, strictly speaking, a special case of the processes defined by Relations~\eqref{eq:markov}, but the tightness results of Section~\ref{sec:aproof} of Appendix concerning occupation times of fast processes can obviously be used. 

Let, for $w{\ge}0$,  $(X^w(t),Z^w(t))$ be the fast processes associated to the model of Definition~\ref{DefFastw}. Proposition~\ref{InvPropFP} shows that $(X^w(t),Z^w(t))$ has a unique invariant distribution $\Pi_w$.  We denote by $(X^w_\infty,Z^w_\infty)$ a random variable with distribution $\Pi_w$.
\begin{proposition}\label{LemLipSimp}
The function 
  \[
w\mapsto \E\left[Z^w_\infty\beta(X^w_\infty)\right]=\int_{\R_+^2} z\beta(x)\Pi_w(\diff x,\diff z)
  \]
  is locally Lipschitz on $\R_+$.
\end{proposition}
\begin{proof}
  %See Section~\ref{sec-A-Simp} of Appendix.
 Assuming $X^w(0){=}Z^w(0){=}0$,
Lemma~\ref{lemma:expofilter} and Definition~\ref{Nbeta} give the relations
  \[
\begin{cases}
\displaystyle  X^w(t)=w\int_0^t e^{-(t-s)}{\cal N}_\lambda(\diff s)=wX^1(t)\\ 
\displaystyle Z^w(t) =B_1\int_0^t e^{-\gamma (t-s)}{\cal N}_\lambda(\diff s){+}B_2\int_0^t e^{-\gamma (t-s)}{\cal P}_2((0,\beta(w X^1(s-)),\diff s).
\end{cases}
\]
the random variable $(X^w(t),Z^w(t))$ is converging in distribution to $(X^w_\infty,Z^w_\infty)$ as well as any of its moments.  Define 
  \[
\Psi_t(w)\steq{def}  \E\left[\beta\left(wX^1(t)\right)\int_0^t e^{-\gamma (t-s)}{\cal P}_2((0,\beta(w X^1(s-)),\diff s)\right],
  \]
for $x$,$y{\ge}0$,
  \begin{multline*}
    \left|\Psi_t(x){-}\Psi_t(y)\right|\\
    \leq     \E\left[\left|\beta\left(xX^1(t)\right){-}\beta\left(y X^1(t)\right)\right|\int_0^t e^{-\gamma (t-s)}{\cal P}_2((0,\beta(x X^1(s-)),\diff s)\right]\\
    {+}\left|\E\left[\beta\left(y X^1(t)\right)\int_0^t e^{-\gamma (t-s)}{\cal P}_2((\beta(x X^1(s-)),\beta(y X^1(s-)),\diff s)\right]\right|.
  \end{multline*}
  We note that $(X^1(t))$ is a functional of ${\cal N}_\lambda$ and is therefore independent of the Poisson process ${\cal P}_2$. We now take care of the two terms of the right-hand side of the last expression.

  For the first term, if $L_\beta$ is the Lipschitz constant of the function $\beta$, we obtain
\begin{multline}\label{A1}
  \E\left[\left|\beta\left(xX^1(t)\right){-}\beta\left(y X^1(t)\right)\right|\int_0^t e^{-\gamma (t-s)}{\cal P}_2((0,\beta(x X^1(s)),\diff s)\right]\\
  =\E\left[\left|\beta\left(xX^1(t)\right){-}\beta\left(y X^1(t)\right)\right|\int_0^t e^{-\gamma (t-s)}\beta(x X^1(s))\diff s\right]\\
  \le L_\beta |x{-}y|\E\left[X^1(t)\int_0^t e^{-\gamma (t-s)}\beta(x X^1(s))\diff s\right],
\end{multline}
and, for the second term, if $|y{-}x|{\le}1$,
\begin{multline}\label{A2}
  \left|\E\left[\beta\left(y X^1(t)\right)\int_0^t e^{-\gamma (t-s)}{\cal P}_2((\beta(x X^1(s)),\beta(y X^1(s)),\diff s)\right]\right|\\
  \le   \E\left[\beta\left(y X^1(t)\right)\int_0^t e^{-\gamma (t-s)}\left|\beta(x X^1(s)){-}\beta(y X^1(s))\right|\diff s\right]\\
    \le  L_\beta|x{-}y| \E\left[\left(\beta(0){+}L_\beta(1{+}x) X^1(t)\right)\int_0^t e^{-\gamma (t-s)}X^1(s))\diff s\right].
\end{multline}
Fubini's Theorem gives the relation,
\[
\E\left[X^1(t)\int_0^t e^{-\gamma (t-s)}\beta(x X^1(s))\diff s\right]=
\int_0^t e^{-\gamma s}\E\left[X^1(t)\beta(x X^1(t{-}s))\right]\diff s,
\]
and the convergence in distribution of  the Markov process $(X^1(t))$ implies the convergence of $(\E\left[X^1(t)\beta(x X^1(t{-}s))\right])$ to a finite limit when $t$ goes to infinity. With Relations~\eqref{A1} and~\eqref{A2} and the expressions of $(X^w(t))$ and $(Z^w(t))$, we deduce that for $x{\ge}0$, there exists a constant $F_x$ independent of $t$ such that
\[
 \left|\E\left[Z^x(t)\beta\left(X^x(t)\right)\right]{-}\E\left[Z^y(t)\beta\left(X^y(t)\right)\right]\right|\leq F_x|x{-}y|
 \]
 holds for all $t{\ge}0$ and $y$ such that $|y{-}x|{\le}1$. We conclude the proof of the proposition by letting $t$ go to infinity. 

\end{proof}
The averaging principle for the simple model, announced in Section~\ref{AVGSubSec} of the introduction can now be stated. 
\begin{theorem}\label{AvgPropSimp}
  If the function $\beta$ is Lipschitz, there exists some $S_0(w_0){>}0$, such that the family of  processes $(W_\eps(t), t{<}S_0(w_0))$ defined by Relation~\eqref{eq:scaledsimple} converges in distribution to $(w(t),t{<}S_0(w_0))$, the unique solution of the ODE
  \[
  \frac{\diff w}{\diff t}(t) =\int_{\R_+^2} z\beta(x)\Pi_{w(t)}(\diff x,\diff z)
  \]
  with $w(0){=}w_0$. 
\end{theorem}
\begin{proof}
For $t{\ge}0$,
  \[
   W_\eps(t)=w_0+\int_0^t Z_\eps(s{-})\eps\mathcal{N}_{\beta/\eps,X_\eps}(\diff s),
   \]
we then proceed as in Section~\ref{section:proof} by using in particular the analogue of Lemma~\ref{ITight} and~\ref{CVX2}.
\end{proof}
An explicit representation of the limiting synaptic weight process can be obtained when linear activation functions.
\begin{proposition}\label{ODESimpAvProp}
If the activation function $\beta$ is such that  $\beta(x){=}\nu{+}\beta_0 x$ and 
\[
\Lambda_2{=}\lambda\beta_0^2B_2\left(\frac{\lambda}{\gamma}{+}\frac{1}{2(\gamma{+}1)}\right),\, \Lambda_1{=}\lambda\beta_0 \left(\frac{B_1}{\gamma{+}1}{+}\frac{\lambda B_1{+}2\nu B_2}{\gamma} \right),  \Lambda_0{=}\frac{\nu}{\gamma}(\lambda B_1{+}\nu B_2),
\]
 then if $\Lambda_2{>}0$,  the asymptotic weight process $(w(t),0{\le}t{<}S_0(w_0))$ of Theorem~\ref{AvgPropSimp} with initial point $w_0{\ge}0$ can be expressed as:
\begin{enumerate}
\item If $\Delta{\steq{def}}\Lambda_1^2{-}4\Lambda_2\Lambda_0{>}0$, then
  \[
w(t)=\frac{\displaystyle s_2(w_0{+}s_1)e^{\sqrt{\Delta} t}{-}s_1(w_0{+}s_2)}{(w_0{+}s_2){-}(w_0{+}s_1)e^{\sqrt{\Delta} t}},
  \quad S_0(w_0){=}\frac{1}{\sqrt{\Delta}}\ln\left(\frac{w_0{+}s_2}{w_0{+}s_1}\right),
  \]
%  \[  w(t)=\frac{\displaystyle s_2(w_0{+}s_1)e^{\Lambda_2s_2 t}{-}s_1(w_0{+}s_2)e^{\Lambda_2s_1 t}}{(w_0{+}s_2)e^{\Lambda_2s_1 t}{-}(w_0{+}s_1)e^{\Lambda_2s_2 t}},  \qquad  0\le t<\frac{1}{\Lambda_2(s_2{-}s_1)}\ln\left(\frac{w_0{+}s_2}{w_0{+}s_1}\right),  \]
  with
  \[
  s_1\steq{def}\frac{\Lambda_1{-}\sqrt{\Delta}}{2\Lambda_2}\text{ and }  s_2\steq{def}\frac{\Lambda_1{+}\sqrt{\Delta}}{2\Lambda_2}.
  \]
\item If $\Delta{=}0$, then
  \[
  w(t)=\frac{2w_0\Lambda_2{+}\Lambda_1}{\Lambda_2(2{-}(2\Lambda_2w_0{+}\Lambda_1)t)}{-}\frac{\Lambda_1}{2\Lambda_2},\quad S_0(w_0){=}\frac{2}{2w_0\Lambda_2{+}\Lambda_1}.
  \]
\item If $\Delta{<}0$, then
  \[
w(t)=\frac{\sqrt{{-}\Delta}}{2\Lambda_2}\left( \tan\left(\frac{1}{2}\sqrt{-\Delta}\cdot t{+}\arctan\left(z_0\right)\right)  {+}\left\lfloor\frac{z_0}{\pi}{+}\frac{1}{2}\right\rfloor \pi\right){-}\frac{\Lambda_1}{2\Lambda_2},  
\]
with
\[
S_0(w_0){=}\frac{2}{\sqrt{-\Delta}}\left(\frac{\pi}{2}{-}\arctan\left(z_0\right)\right)\text{ and }
z_0\steq{def}\frac{2w_0\Lambda_2{+}\Lambda_1}{\sqrt{{-}\Delta}}.
\]
\end{enumerate}
\end{proposition}
It should be noted that under the conditions of this proposition, this model always exhibits a blow-up phenomenon.
\begin{proof}
  The SDEs give the relation
\begin{multline*}
\diff X^wZ^w(t) = -(\gamma{+}1)X^wZ^w(t)\diff t \\ + (wZ^w(t{-}){+}B_1w{+}B_1X^w(t{-})){\cal N}_{\lambda}(\diff t) + B_2X^w(t{-}){\cal N}_{\beta,X^w}(\diff t).
\end{multline*}
If the initial point has the same distribution as $(X^w_\infty,Z^w_\infty)$, by integrating and by  taking the expected valued of this SDE, we obtain the identity
\[
(\gamma{+}1)\E\left[X^w_\infty Z^w_\infty\right]=\lambda wB_1+(\lambda B_1{+}\nu B_2)\E\left[X^w_\infty\right]+\beta_0 B_2\E\left[(X^w_\infty)^2\right]{+}\lambda w\E[Z^w_\infty].
\]
With Relations~\eqref{SNnRS}, we have
\[
\E\left[X^w_\infty\right]=\lambda w,\quad \E\left[(X^w_\infty)^2\right]=\left(\lambda^2 {+}\frac{\lambda}{2}\right)w^2,
\]
and, similarly,
\[
\gamma\E[Z^w_\infty]=\lambda B_1{+}B_2(\nu{+}\beta_0\E[X^w_\infty])=
\lambda B_1{+}B_2(\nu{+}\beta_0\lambda w).
\]
By using Theorem~\ref{AvgPropSimp}, with these identities,  we obtain that $(w(t))$ satisfies the ODE
\begin{equation}\label{ODESimpAv}
  \frac{\diff w}{\diff t}(t) =\E\left[Z^w_\infty\left(\nu{+}\beta_0 X^w_\infty\right)\right]=\Lambda_2w^2{+}\Lambda_1 w{+}\Lambda_0,
\end{equation}
on its domain of definition. We conclude the proof with trite calculations.  
\end{proof}

\printbibliography

\newpage

\appendix

\begin{figure}[ht]  \fontsize{8pt}{10pt}\selectfont\centering{\makebox[\textwidth][c]{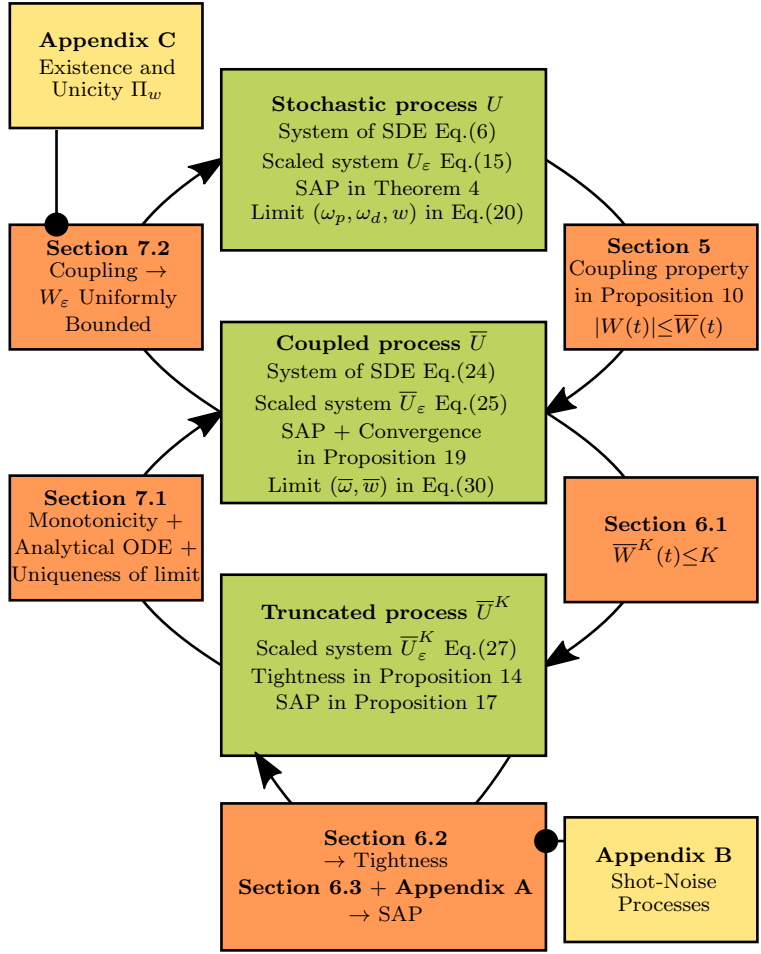}}\caption{Graphical Representation of the Steps of the Proof  of  Averaging Principles}\label{fig:proof}
\end{figure}

\section{Proofs of Technical Results for Occupation Times}
\label{sec:aproof}
\subsection{Proof of Lemma~\ref{CVX2}}
Denote, for $t{\ge}0$, $a$, $b$, $c{\in}\R_+$, and $\eps{>}0$,
\[
L_{\eps}(t) \steq{def} a\Dom{X}_{\eps}^K(t){+}b \Dom{Z}_{\eps}^K(t){+}c\Dom{X}_{\eps}^K(t)\Dom{Z}_{\eps}^K(t).
  \]
Let $G$ be a continuous bounded Borelian function.
From Proposition~\ref{ITight}, we can extract a sub-sequence $(\eps_n)$ such that, for the convergence in distribution
 \[
    \lim_{n\to+\infty} \left(\int_0^tL_{\eps_{n}}(u)G\left(\Dom{X}_{\eps_{n}}^K(u),\Dom{Z}_{\eps_{n}}^K(u)\right) \diff u\right)=(L(t)),
    \]
where $(L(t))$ is a  continuous \cadlag process.

 We will now prove that the process $(L(t))$ is such that,
 \[
 (L(t)) = \left(\int_0^t\int_{\R^2} (ax{+}bz{+}cxz)G(x,z)\Dom{\nu}^K(\diff s, \diff x,\diff z)\right)
 \]
holds almost-surely for $t{\in}[0,T]$.

For $A{>}0$,  the convergence of $(\Dom{\nu}^K_{\eps_n})$ to $\Dom{\nu}^K$ gives the convergence in distribution,
 \begin{multline}\label{eqqj}
  \lim_{n\to+\infty} \left(\int_0^t A{\wedge} L_{\eps_{n}}(s)G\left(\Dom{X}_{\eps_{n}}^K(s),\Dom{Z}_{\eps_{k_n}}^K(s)\right)\diff s\right)
\\  =\left(\int_0^t\int_{\R_+^2} A{\wedge}  (ax{+}bz{+}cxz) G(x,z)\Dom{\nu}^K(\diff s, \diff x,\diff z)\right).
 \end{multline}
Using again the upper-bound, Relation~\eqref{eqpp2}, for $(\Dom{X}^K_\eps(t))$ with  Relation~\eqref{eqBouExp},  and Proposition~\ref{R4Mom} for $R_{\eps}$, we obtain that
 \[
C_{L}\steq{def} \sup_{\substack{0{<}\eps{<}1\\0{\leq}t{\leq}T}} \E\left[L_\eps(s)^2\right]{<}+\infty,
 \]
hence, for $\eta{>}0$,
 \begin{multline*}
   \P\left(\int_0^T (L_\eps(s){-}A)^+\diff s{\ge}\eta \right)\leq \frac{1}{\eta}\int_0^T\E\left[(L_\eps(s){-}A)^+\right]\diff s
 \\ \leq \frac{1}{\eta A}\int_0^T\E\left[L_\eps(s)^2\right]\diff s\le \frac{C_{L}T}{\eta A}.
 \end{multline*}

Since $G$ is bounded,  with the elementary relation $x{=}x{\wedge}A{+}(x{-}A)^+$, $x{\ge}0$, then, for $n{\ge}1$,
 \begin{multline}\label{eqqk}
 \P\left(\sup_{0\le t\le T}\left|\int_0^tL_{\eps_{n}}(u)G\left(\Dom{X}_{\eps_{n}}^K(u),\Dom{Z}_{\eps_{n}}^K(u)\right) \diff u\right.\right.\\\left.\left.{-}\int_0^t A{\wedge}L_{\eps_{n}}(u)G\left(\Dom{X}_{\eps_{n}}^K(u),\Dom{Z}_{\eps_{n}}^K(u)\right)\diff u\right|{\ge}\eta \right)\leq\frac{C_{L} T}{\eta A}\|G\|_{\infty}.
 \end{multline}
 For any $A{>}0$ and $n{\ge}1$, Cauchy-Schwartz's inequality gives the relation
\[
  \E\left[\int_0^T A{\wedge}L_{\eps_{n}}(u)G\left(\Dom{X}_{\eps_{n}}^K(u),\Dom{Z}_{\eps_{n}}^K(u)\right)\diff u\right]
  \leq \sqrt{C_L}T\|G\|_{\infty}.
  \]
  With Relation~\eqref{eqqj} and the fact that the left-hand side of~\eqref{eqqj} has a bounded second moment, by letting $n$ go to infinity, we get the inequality
  \[
  \E\left[\int_0^T\int_{\R_+^2} A{\wedge}  (ax{+}bz{+}cxz) G(x,z)\Dom{\nu}^K(\diff s, \diff x,\diff z)\right]  \leq \sqrt{C_L}T\|G\|_{\infty}.
  \]
  By letting $A$ go to infinity, the monotone convergence theorem shows that
\begin{multline*}
\lim_{A\to+\infty} \E\left[\int_0^T\int_{\R_+^2} A{\wedge}(ax{+}bz{+}cxz) G(x,z)\Dom{\nu}^K(\diff s, \diff x,\diff z)\right] \\= \E\left[\int_0^T\int_{\R_+^2} (ax{+}bz{+}cxz) G(x,z)\Dom{\nu}^K(\diff s, \diff x,\diff z)\right] \le \sqrt{C_L}T\|G\|_{\infty} . <{+}\infty.
\end{multline*}
With Relation~\eqref{eqqk} and the integrability properties proven just above, we obtain that, for $\eps{>}0$, there exists $n_0$ such that if $n{\ge}n_0$, then the relation
\begin{multline*}
 \P\left(\sup_{0\le t\le T}\left|\int_0^tL_{\eps_{n}}(u)G\left(\Dom{X}_{\eps_{n}}^K(u),\Dom{Z}_{\eps_{n}}^K(u)\right) \diff u\right.\right.\\\left.\left.{-}\int_0^T\int_{\R_+^2} (ax{+}bz{+}cxz) G(x,z)\Dom{\nu}^K(\diff s, \diff x,\diff z)\right|{\ge}\eta \right)\leq\eps
\end{multline*}
holds. Lemma~\ref{CVX2} is proved.

\subsection{Proof of Lemma~\ref{OccMeasProp}}
Following~\citet{papanicolalou_martingale_1977} and~\citet{karatzas_averaging_1992}, we first show that there exists an optional process  $(\Dom{\Gamma}^{K}_s)$, with values in the set of probability distributions on $\R_+^2$ such that, almost surely, for any bounded Borelian function $G$ on $\R_+{\times}\R_+^2$,
\begin{equation}\label{GammP}
\int_{\R_+{\times}\R_+^2} G(s,x,z)\Dom{\nu}^{K}(\diff s, \diff x, \diff z)=\int_{\R_+{\times}\R_+^2} G(s,x,z)\Dom{\Gamma}_s^{K}(\diff x, \diff z)\,\diff s.
\end{equation}
Recall that the optional $\sigma$-algebra is the smallest $\sigma$-algebra containing adapted \cadlag processes.
See Section~VI.4 of~\citet{rogers_diffusions_2000} for example.
This is a simple consequence of Lemma~1.4 of~\citet{karatzas_averaging_1992} and the fact that, due to Relation~\eqref{OddMDef}, the measure  $\Dom{\nu}^{K}(\diff s,\R^2)$ is the Lebesgue measure on $[0,T]$.

Let $f{\in}{\cal C}_b^1(\R_+^2)$ be a bounded ${\cal C}^1$-function on $\R_+^2$ with bounded partial derivatives, we have the relation
\begin{multline}\label{jka1}
\eps f\left(\Dom{X}^{K}_\eps(t),\Dom{Z}^{K}_\eps(t)\right)=\eps f(\Dom{x}_0,\Dom{z}_0){+}\eps \Dom{M}_\eps^f(t)\\{+}\int_0^t B^F_{K{\wedge}\Dom{W}^{K}_\eps(s)}(f)\left(\Dom{X}^{K}_\eps(s),\Dom{Z}^{K}_\eps(s)\right) \diff s,
\end{multline}
where, for $t{\ge}0$,  if $(\Dom{V}^K_\eps(s)){\steq{def}}(\Dom{X}^{K}_\eps(s),\Dom{Z}^{K}_\eps(s))$,
\begin{multline*}
\Dom{M}_{\eps}^f(t)\steq{def}\hspace{-2mm}\int_0^t\hspace{-2mm}\left(f\left(\Dom{V}^{K}_\eps(s-){+}\left(K{\wedge}\Dom{W}^{K}_\eps(s),1\right)\right){-}f\left(\Dom{V}^{K}_\eps(s{-})\right)\right)\left[{\cal N}_{\lambda/\eps}(\diff s){-}\frac{\lambda}{\eps}\diff s\right]\\
  {+}\int_0^t\left(f\left(\Dom{V}^{K}_\eps(s-){+}(0,1)\right){-}f\left(\Dom{V}^{K}_\eps(s-)\right)\right)\left[{\cal N}_{\Dom{\beta}/\eps,\Dom{X}^{K}_\eps}(\diff s){-}\frac{\Dom{\beta}\left(\Dom{X}_{\eps}^{K}(s)\right)}{\eps}\diff s\right],
\end{multline*}
Proposition~\ref{ITight} shows that the martingale $(\eps\Dom{M}_{\eps}^f(t))$ is converging in distribution to~$0$ as $\eps$ goes to~$0$.

Relation~\eqref{jka1} gives therefore the convergence in distribution
\[
\lim_{\eps\to 0} \left(\int_0^t B^F_{K{\wedge}\Dom{W}^{K}_\eps (s)}(f)\left(\Dom{X}_\eps^{K}(s),\Dom{Z}_\eps^{K}(s)\right) \diff s\right){=}0.
\]

The convergence in distribution of $(\Dom{\Omega}^{K}_{\eps_n}(t), \Dom{W}^{K}_{\eps_n}(t), \Dom{\nu}_{\eps_n}^{K})$,  Proposition~\ref{CVX2} and Relation~\eqref{GammP} give that, for any $f$ in ${\cal C}_b^1(\R_+^2)$, the relation
\begin{equation}\label{EqGen}
\left(\int_0^t \int_{\R_+^2} B^F_{K{\wedge}\Dom{w}^{K}(s)}(f)(x,z)\Dom{\Gamma}_s^{K}(\diff x,\diff z) \diff s,\, 0{\leq}t{\leq}T\right)=(0,\, 0{\leq}t{\leq}T).
\end{equation}
holds with probability $1$.

Let $(f_n)$ be a dense countable sequence in ${\cal C}_b^1(\R_+^2)$ and ${\cal E}_1$ be the event,  where Relation~\eqref{EqGen} holds for all $f{=}f_n$, $n{\ge}1$.
Note that  that $\P\left({\cal E}_1\right){=}1$.
On ${\cal E}_1$, there exists a  (random) subset $S_1$ of $[0,T]$ with Lebesgue measure $T$ such that
\[
 \int_{\R_+^2} B^F_{K{\wedge}\Dom{w}^{K}(s)}(f_n)(x,z)\Dom{\Gamma}_s^{K}(\diff x,\diff z)=0, \forall s{\in}S_1 \text{ and }\forall  n{\ge}1,
 \]
 and, consequently,
\[
 \int_{\R_+^2} B^F_{K{\wedge}\Dom{w}^{K}(s)}(f)(x,z)\Dom{\Gamma}_s^{K}(\diff x,\diff z)=0, \forall s{\in}S_1 \text{ and }\forall f{\in} {\cal C}_b^1(\R_+^2).
 \]
 By Proposition~\ref{InvPropFP}, for $s{\in}S_1$, the probability distribution $\Dom{\Gamma}_s^{K}$ is the invariant distribution $\Pi_{K{\wedge}\Dom{w}^{K}(s)}$.
Lemma~\ref{OccMeasProp} is proved.
\section{Shot-Noise Processes}\label{ShotSec}
This section presents several technical results on {\em shot-noise} processes which are crucial for the proof of Theorem~\ref{theorem:homog}.
See~\citet{Schottky}, \citet{Rice} and~\citet{gilbert_amplitude_1960} for an introduction.

\subsection{A Scaled Shot-Noise Process}\label{ScSN}
Recall that  $(S^x_{\eps}(t))$, with initial point $x{\ge}0$, has been introduced by Definition~\ref{SNPdef}.
We will have the following conventions,
\[
\left(S_{\eps}(t)\right)\steq{def}\left(S^0_{\eps}(t)\right), \left(S^x(t)\right)\steq{def}\left(S^x_{1}(t)\right) \text{ and }
\left(S(t)\right)\steq{def}\left(S^0_{1}(t)\right).
\]

The process $(S(t))$ is in fact the standard shot-noise process of Lemma~\ref{lemma:expofilter} associated to the Poisson process ${\cal N}_\lambda$, for $t{\ge}0$, 
\begin{equation}\label{eq:slambda0}
S(t)=\int_0^t e^{-(t-s)}{\cal N}_\lambda(\diff s)\steq{dist}\int_0^t e^{-s}{\cal N}_\lambda(\diff s).
\end{equation}
In particular $(S(t))$ is a stochastically non-decreasing process, i.e.\ for $y{\ge}0$ and $s{\le}t$,
\begin{equation}
\label{StocMon}
\P\left(S(s){\ge}y\right)\le \P\left(S(t){\ge}y\right).
\end{equation}
A classical formula for Poisson processes, see Proposition~1.5 of~\citet{robert_stochastic_2003} for example, gives the relation, for $\xi{\in}\R$,
\begin{equation}\label{eqfg1}
\E\left[e^{\xi S(t)}\right]=\exp\left({-}\lambda\int_0^t \left(1{-}\exp\left(\xi e^{-s}\right)\right)\diff s\right),
\end{equation}
in particular $\E\left[S^x(t)\right]{=}x\exp(-t){+}\lambda\left(1{-}\exp({-}\lambda t)\right)$.
It also  shows that $(S^x(t))$ is converging in distribution to $S(\infty)$ such that,
\[
\E\left[e^{\xi S(\infty)}\right]=\exp\left({-}\lambda\int_0^{+\infty} \left(1{-}\exp\left(\xi e^{-s}\right)\right)\diff s\right){<}{+}\infty.
\]

It is easily seen that $(S^x_{\eps}(t)){\steq{dist}} (S^x(t/\eps))$and thus with Relation~\eqref{eqSN}, $(S^x_{\eps}(t))$ can be represented as, for $t{\ge}0$,
\begin{equation}\label{eqU}
  S^x_{\eps}(t)\steq{def} xe^{-t/\eps}{+} \int_0^t e^{-(t-s)/\eps}{\cal N}_{\lambda/\eps}(\diff s)
  = xe^{-t/\eps}{+} S_{\eps}(t).
\end{equation}

We remind here the results of Proposition~\ref{lemma:Slambda}, that will be proved in the following paragraph.
For $\xi{\in}\R$ and $x{\geq}0$,  the convergence in distribution of the processes
  \[
\lim_{\eps\searrow 0} \left(\int_0^t e^{\xi S^x_{\eps}(u)}\diff u\right)=\left(\E\left[e^{\xi S(\infty)}\right]t\right)
  \]
 holds, and
\[
\sup_{\substack{0{<}\eps{<}1\\0{\leq}t{\leq}T}} \E\left[e^{\xi S_{\eps}(t)}\right] {<}{+}\infty.
\]

\begin{proof}[Proof of Proposition~\ref{lemma:Slambda}]
Let $T_1$ and $T_2$ be two stopping times bounded by $N$, $\theta{>}0$, and verifying $0{\le}T_2{-}T_1{\le}\theta$.
Using Relation~\eqref{eqU} and the strong Markov property of Poisson processes, we have that
\begin{multline*}
  \E\left[\int_{T_1}^{T_2} e^{\xi S^x_{\eps}(u)}\diff u\right]=  \eps\E\left[\int_{T_1/\eps}^{T_2/\eps} e^{\xi S^x(u)}\diff u\right]\\% =\eps \E\left.\left(\E\left(\int_{T_1/\eps}^{T_2/\eps} e^{\xi S^x(u)}\diff u\right| {\cal F}_{T_1/\eps}\right)\right)\\
  =\eps \E\left[\int_{0}^{(T_2-T_1)/\eps} e^{\xi S^{S^x(T_1/\eps)}(u)}\diff u\right]\le\eps \E\left[e^{\xi S^x(T_1/\eps)}\E\left[\int_{0}^{\theta/\eps} e^{\xi S(u)}\diff u\right]\right]
\\  \le \theta e^{\xi x}\E\left[e^{\xi S(N/\eps)}\right]\E\left[e^{\xi S(\infty)}\right]
  \le \theta e^{\xi x}\E\left[e^{\xi S(\infty)}\right]^2,
\end{multline*}
holds, by stochastic monotonicity of $(S(t))$ of Relation~\eqref{StocMon}.

Aldous' Criterion, see Theorem~VI.4.5 of~\citet{jacod_limit_1987} gives that the family of processes
\[
\left(\int_0^te^{\xi S_{\eps}(u)}\diff u\right),
\]
is tight when $\eps$ goes to $0$.
For $p{\ge}1$ and a fixed vector $(t_i){\in}\R_+^p$, the ergodic theorem for the Markov process $(S(t))$ gives the almost-sure convergence of
\begin{multline*}
\lim_{\eps\to 0}\left(\int_0^{t_i}e^{\xi S_{\eps}(u)}\diff u, i=1,\ldots,p\right)=\lim_{\eps\to 0}\left(\eps\int_0^{t_i/\eps}e^{\xi S(u)}\diff u,i=1,\ldots,p\right)\\=\left(\E\left[e^{\xi S(\infty)}\right]t_i,i=1,\ldots,p\right).
\end{multline*}
Hence, due to the tightness property, the convergence also holds in distribution for the processes
\[
\lim_{\eps\to 0}\left(\int_0^te^{\xi S_{\eps}(u)}\diff u\right)=\left(\E\left[e^{\xi S(\infty)}\right]t\right).
\]

The last part is a direct consequence of the identity $(S^x_{\eps}(t)){\steq{dist}} (S^x(t/\eps))$,  and  of Relation~\eqref{eqfg1}, which gives
\begin{multline*}
\E\left[e^{\xi S_{\eps}(t)}\right]=\exp\left({-}\lambda\int_0^{t/\eps} \left(1{-}\exp\left(\xi e^{-s}\right)\right)\diff s\right)
\\\le \exp\left({-}\lambda\int_0^{+\infty} \left(1{-}\exp\left(\xi e^{-s}\right)\right)\diff s\right).
\end{multline*}
The proposition is proved. 
\end{proof}

\subsection{Interacting Shot-Noise Processes}
\label{secsec:interacting}
Recall that $(R_{\eps}(t))$ defined by Relation~\eqref{Reps} is a shot-noise process with intensity equal to the shot-noise process $(S_{\eps}(t))$.

We start with a simple result on moments of functionals of Poisson processes. 
\begin{lemma}\label{MomPois}
  If ${\cal Q}$ is a Poisson point process on $\R_+$ with a positive Radon intensity measure $\mu$ and $f$ is a Borelian function such that
  \[
  I_k(f)\steq{def}\int f(u)^k\mu(\diff u) < {+}\infty,  1{\le}k{\le} 4,
  \]
  then
  \[
  \E\left[\left(\int f(u) {\cal Q}(\diff u)\right)^4\right]=\left(\rule{0mm}{4mm}I_4{+}6I_{1}^2I_{2}{+}4I_{1}I_{3}{+}3I_{2}^2{+}I_{1}^4\right)(f).
  \]
\end{lemma}
\begin{proof}

It is enough to prove the inequality for non-negative bounded Borelian functions $f$ with compact support on $\R_+$.
  
The formula for the Laplace transform of Poisson point processes, see Proposition~1.5 of~\citet{robert_stochastic_2003},  gives for $\xi{\ge}0$,
  \[
  \E\left[\exp\left(\xi \int_0^{+\infty} f(u) {\cal Q}(\diff u)\right)\right]=
  \exp\left(\int_0^{+\infty} \left(e^{\xi f(u)}{-}1\right)\mu(\diff u)\right).
  \]
The proof is done in a straightforward way by differentiating the last identity with respect to $\xi$ four times and then set $\xi{=}0$.
\end{proof}

\begin{proposition}\label{R4Mom}
The inequality
\[
\sup_{\eps{\in}(0,1), t\geq 0}\E\left[R_\eps(t)^4\right] <{+}\infty
\]
holds.
\end{proposition}
\begin{proof}
  Denote, for $t{\ge}0$,
  \[
  J_{k,\eps}(t)\steq{def} \int_0^t  e^{-\gamma k (t-u)/\eps}\frac{S_{\eps}(u)}{\eps}\diff u,
  \]
  the identity $(S_{\eps}(t)){\steq{dist}} (S(t/\eps))$ and Relation~\eqref{eq:slambda0} coupled with Fubini's Theorem give the relations
\begin{multline*}
J_{k,\eps}(t){=}\int_0^{t/\eps}  e^{-\gamma k (t/\eps-u)}S(u)\diff u
{=}\frac{1}{\gamma k{-}1}\int_0^{t/\eps} \left(e^{-(t/\eps-v)}{-}e^{-\gamma k (t/\eps-v)}\right) {\cal N}_\lambda(\diff v)\\
\steq{dist} \frac{1}{\gamma k{-}1}\int_0^{t/\eps} \left(e^{-v}{-}e^{-\gamma kv}\right) {\cal N}_\lambda(\diff v)
\le\frac{1}{|\gamma k{-}1|} \Dom{J}_{k}
\end{multline*}
with
\[
\Dom{J}_{k}\steq{def} \int_0^{{+}\infty} \left(e^{-k \gamma v}{+}e^{-v}\right) {\cal N}_\lambda(\diff v).
\] 
Relation~\eqref{eqSN} applied to $R_\eps(t)$ gives
\[
R_\eps(t)=\int_0^t e^{-\gamma (t-u)/\eps} {\cal P}_2\left(\left(0,\frac{S_{\eps}(u{-})}{\eps}\right],\diff u\right).
\]
The quantity $S_{\eps}(u)$ is a functional of the point process ${\cal P}_1$ and is therefore independent of the Poisson point process ${\cal P}_2$. Lemma~\ref{MomPois} gives therefore that
\begin{multline*}
 \E\left[R_\eps(t)^4\mid {\cal P}_1\right]=
  J_{4,\eps}(t){+}6J_{1,\eps}(t)^2J_{2,\eps}(t){+}4J_{1,\eps}(t)J_{3,\eps}(t){+}3J_{2,\eps}(t)^2{+}J_{1,\eps}(t)^4,
\end{multline*}
hence,
\[
\E\left[R_\eps(t)^4\right] \le  \E\left[\rule{0mm}{4mm}
\frac{\Dom{J}_{4}}{|4\gamma{-}1|}{+}\frac{6\Dom{J}_{1}^2\Dom{J}_{2}}{|\gamma{-}1|^2|2\gamma{-}1|}{+}\frac{4\Dom{J}_{1}\Dom{J}_{3}}{|\gamma{-}1||3\gamma{-}1|}{+}\frac{3\Dom{J}_{2}^2}{|2\gamma{-}1|^2}{+}\frac{\Dom{J}_{1}^4}{|\gamma{-}1|^4}\right].
\]
Again with Proposition~\ref{lemma:Slambda} we obtain that, for $k{\ge}1$, the variable $\Dom{J}_{k}$ has finite moments of all orders, therefore by Cauchy-Shwartz' Inequality, the right-hand side of the last inequality is finite.
The proposition is proved.
\end{proof}

\subsection{Tightness Properties}\label{TiSecOM}
We can now prove  Proposition~\ref{Fam1} of Section~\ref{section:theorems}. 
\begin{proof}[Proof of Proposition~\ref{Fam1}]
With the criterion of the modulus of continuity, see~\citet{billingsley_convergence_1999}, it is enough to prove the tightness of the second family of processes. Indeed, by Cauchy-Schwartz' Inequality, for $0{\le}s{\le}t$,
\[
\int_s^t  H_{\eps}(u)\,\diff u\leq \sqrt{t{-}s}\sqrt{\int_s^t  H_{\eps}(u)^2\,\diff u},
\]
and
\[
\int_s^t R_{\eps}(u)S_{\eps}(u)\,\diff u\le
\sqrt{\int_s^t R_{\eps}(u)^2\,\diff u}
\sqrt{\int_s^t S_{\eps}(u)^2\,\diff u}.
\]
Again with Cauchy-Schwartz' Inequality, we have
\[
\int_s^t  H_{\eps}(u)^2\,\diff u\leq \sqrt{t{-}s} \sqrt{\int_s^t H_{\eps}(u)^4\,\diff u},
\]
hence, with Relation~\eqref{eqBouExp} for $H_\eps{=}S_{\eps}$ and Proposition~\ref{R4Mom} for $H_\eps{=}R_{\eps}$, there exists a constant $C$ independent of $\eps$ and $s$, $t$ such that
\[
\E\left[\left(\int_s^t  H_{\eps}(u)^2\,\diff u\right)^{2} \right]\leq C (t{-}s)^2.
\]
Kolmogorov-\v{C}entsov's Criterion, see Theorem~2.8 and Problem~4.11 of~\citet{Karatzas}, implies that the family of variables
\[
\left(\int_0^t  H_{\eps}(u)^2\,\diff u\right)
\]
is tight. The proposition is proved. 
  
\end{proof}

\section{Equilibrium of Fast Processes}\label{section:invariant}
For $w{\in}K_W$, recall that the Markov process $(X^w(t),Z^w(t))$ of Definition~\ref{DefFastw} is such that
\begin{align}
\diff X^w(t) &\displaystyle = {-}X^w(t)\diff t{+}w\mathcal{N}_{\lambda}(\diff t){-}g\left(X^w(t{-})\right)\mathcal{N}_{\beta,X^w}\left(\diff t\right)\label{MarkXw}\\
\diff Z^w(t) &\displaystyle = ({-}\gamma{\odot}Z^w(t){+}k_0)\diff t{+}k_1(Z^w(t{-}))\mathcal{N}_{\lambda}(\diff t){+}k_2(Z^w(t{-}))\mathcal{N}_{\beta,X^w}(\diff t).\label{MarkZw}
\end{align}

\begin{proposition}\label{InvPropFP}
 Under the conditions of Sections~\ref{SecPost} and~\ref{SecZ}, the Markov process $(X^w(t),Z^w(t))$ solution of the SDEs~\eqref{MarkXw} and~\eqref{MarkZw} has a unique invariant distribution $\Pi_w$, i.e. the unique probability distribution $\mu$ on $\R{\times}\R_+^\ell$ such that
\begin{equation}\label{muOm}
\croc{\mu,B_w^F(f)}=\int_{\R{\times}\R_+^\ell} B_w^F(f)(x,z)\mu(\diff x,\diff z)=0,
\end{equation}
for any $f{\in}{\cal C}_b^1(\R{\times}\R_+^\ell)$, where $B^F_w$ is the operator defined by Relation~\eqref{BFGen}.
  \end{proposition}

\begin{proof}
We denote by $(X^w_n,Z^w_n)$  the  embedded Markov chain of the Markov process $(X^w(t),Z^w(t))$, i.e. the sequence of states visited by $(X^w(t),Z^w(t))$ after each jump, of either $\mathcal{N}_{\lambda}$ or $\mathcal{N}_{\beta,X}$.

The proof of the proposition is done in three steps.
We first show that the return time of $(X^w(t),Z^w(t))$ to a compact set of $\R{\times}\R_+^\ell$ is integrable.
Then we prove that the Markov chain $(X^w_n,Z^w_n)$ is Harris ergodic, and consequently that it has a unique invariant measure.
For a general introduction on Harris Markov chains, see~\citet{nummelin_general_2004,meyn_markov_1993}.
Finally, the proof of the proposition uses the classical framework of stationary point processes.

\subsection{Integrability of Return Times to a Compact Subset}
Suppose that $w{\ge}0$.
The conditions of Section~\ref{SecPost} on the functions $\beta$ and $g$, and Relation~\eqref{MarkXw} show that $X^w(t){\ge}{-}c_0$, for all $t{\ge}0$, if $X^w(0){\ge}{-}c_0$, with $c_0{=}c_\beta{+}c_g$. The state space of the Markov process $(X^w(t),Z^w(t))$  can be taken as ${\cal S}{\steq{def}}[{-}c_0,{+}\infty){\times}\R_+^\ell$.

Define, for $(x,z){\in}{\cal S}$ and $0{<}a{\le}1$,
\[
H(x,z)\steq{def} x{+}a \|z\|, \text{ with } \|z\|\steq{def} \sum_{i=1}^{\ell} z_i,
\]
we get that
\begin{multline*}
B_w^F(H)(x,z)= {-}x+\left({-}a\sum_{i=1}^{\ell} \gamma_{i}z_{i}{+}k_{0,i}\right)+\lambda\left(w{+}a\sum_{i=1}^{\ell} k_{1,i}(z)\right) \\+
\beta(x)\left({-}g(x){+}a\sum_{i=1}^{\ell} k_{2,i}(z)\right)
\end{multline*}
hence, with the assumptions of Section~\ref{SecZ} and~\ref{SecPost} on the function $k_.$ and $\beta$, and $a\leq 1$,
\begin{multline*}
B_w^F(H)(x,z)\le {-}x{-}a\underline{\gamma}\|z\| +\ell C_k +\lambda(w{+}\ell a C_k){+}C_\beta\left(1 + x\right)\ell a C_k\\\le
(\ell a C_\beta C_k{-}1)x-a\underline{\gamma}\|z\|+\left(\ell C_k {+} \lambda w {+} \lambda \ell C_k {+} \ell C_\beta C_k\right)\\
\le
(\ell a C_\beta C_k{-}1)x{-}a\underline{\gamma}\|z\|{+}C,
\end{multline*}
where $\underline{\gamma}{>}0$ is the minimum of the coordinates of $\gamma$ and $C$ is a constant independent of $x$, $z$ and $a$.
We fix $0{<}a{\leq}1$ sufficiently small so that $\ell aC_\beta C_k{<}1$ and $K{>}c_0$ such that
\[
C{<}\underline{\gamma}{K}/{2}{-}1 \text{ and } C{<}(1{-}\ell aC_\beta C_k){K}/{2}{-}1.
\]
If $H(x,z){>}K$ then $\max(x,a\|z\|){>}K/2$ and therefore $B_w^F(H)(x,z){\le}{-}1$, $H$ is therefore a Lyapounov function for $B^F_w$.
One deduces that the same result holds for the return time of Markov chain, $(X^w_n,Z^w_n)$ in the set $I_K=\{(x,z):H(x,z){\leq}K\}$.

\subsection{Harris Ergodicity of $(X^w_n,Z^w_n)$}
Proposition~5.10 of~\citet{nummelin_general_2004} is used to show that $I_K$ is a recurrent set.
A {\em regeneration property} would be sufficient to conclude.
In particular, we can prove that $I_K$ is a {\em small set}, that is, there exists some positive, non-trivial, Radon measure $\nu$ on ${\cal S}$ such that,
\begin{equation}\label{Small}
  \P_{(x_0, z_0)}\left(\rule{0mm}{4mm}(X_1^w,Z_1^w){\in}S\right)\ge \nu(S),
\end{equation}
for any Borelian subset $S$ of ${\cal S}$ and all $(x_0, z_0){\in}I_K$.

We denote by $s_1$, resp. $t_1$,  the first instant of ${\cal N}_\lambda$, resp. of ${\cal N}_{\beta,X^w}$,  then, for $(X_0^w,Z_0^w){=}(x_0, z_0){\in}I_K$, by using the deterministic differential equations between jumps, we get
\[
\P_{(x_0, z_0)}\left(s_1{<} t_1\right)
  {=} \E\left[\exp\left({-}\hspace{-2mm}\int_0^{s_1}\hspace{-3mm}\beta(x_0\exp({-}s)\right)\,\diff s\right]\ge
  \E\left[\exp({-}c_\beta^1s_1)\right]{=}p_0{\steq{def}}\frac{\lambda}{\lambda{+}c_\beta^1},
       \]
    since $\beta$ is bounded by some constant $c_\beta^1$ on the interval $[{-}c_0,K]$.

In the following argument, we restrict $X$ to be non-negative, the extension to $[{-}c_0,+\infty]$ is straightforward.
For ${\cal A}{=}[0,A]{\in}{\cal B}(\R_+)$ and ${\cal B}{=}[0,B]{\in}{\cal B}(\R_+^\ell)$, from Equations~\eqref{MarkXw} and~\eqref{MarkZw}, we obtain the relation
\begin{multline*}
  \P_{(x_0, z_0)}\left(\rule{0mm}{4mm}(X_1^w,Z_1^w){\in}{\cal A}{\times}{\cal B}\right)\ge p_0\P\left(\rule{0mm}{4mm}(X_1^w,Z_1^w){\in}{\cal A}{\times}{\cal B}\mid s_1{<}t_1\right)\\
  = p_0\P\left(\left.\substack{\displaystyle x_0e^{-s_1}{+}w {\in}{\cal A},\\\displaystyle (z_0{-}k_0){\odot}e^{-\gamma_{i} s_1}{+}k_0{+}k_1\left((z_0{-}k_0){\odot}e^{-\gamma_{i} s_1}{+}k_0\right){\in}{\cal B}}\right| s_1{<}t_1\right)\\
  \geq \P\left(\left.\substack{\displaystyle x_0e^{-s_1}{+}w {\in}{\cal A},\\\displaystyle H\left((z_0{-}k_0){\odot}e^{-\gamma_{i} s_1}{+}k_0\right){\in}{\cal B}}\right| s_1{<}t_1\right)\\
  =p_0 \P\left(\substack{\displaystyle x_0e^{-\bar{s}_1}{+}w {\in}{\cal A},\\\displaystyle H\left((z_0{-}k_0){\odot}e^{-\gamma_{i} \bar{s}_1}{+}k_0\right){\in}{\cal B}}\right),
   %  = p_0\P\left[\substack{\displaystyle \bar{s}_1{\in}\log(x_0){-}\log(A{-}w),\\\displaystyle  \bar{s}_1\gamma_{i}{\in}\log|z_0{-}k_0|{-}\log|B{-}k_0{-}C_k|)}\right]
\end{multline*}
where $H(z){=} z{+}k_1(z)$,  $\bar{s}_1{\steq{dist}}(s_1{\mid}s_1{\leq}t_1)$.
By using the fact that $k_1$ is in ${\cal C}_b^1(\R_+^\ell,\R_+^\ell)$ by the conditions of Section~\ref{SecZ} and in the same way as Example of Section~4.3.3 page~98 of~\citet{meyn_markov_1993}, we can prove that the random variable
\[
\left(x_0e^{-\bar{s}_1}{+}w,H\left((z_0{-}k_0){\odot}e^{-\gamma_{i} \bar{s}_1}{+}k_0\right)\right)
\]
has a density,   uniformly bounded below by a positive function $h$ on $\R_+{\times}\R_+^{\ell}$, so that
\[
\P_{(x_0, z_0)}\left(\rule{0mm}{4mm}(X_1^w,Z_1^w){\in}A{\times}B\right)\ge\int_{A{\times}B} h(x,z)\,\diff x\diff z, \forall A{\in}{\cal B}(\R_+), B{\in}{\cal B}(\R_+^\ell),
\]
for all $(x_0,z_0){\in}I_K)$. This relation is then extended to all Borelian subsets $S$ of ${\cal S}$, so that Relation~\eqref{Small} holds.
Proposition~5.10 of~\citet{nummelin_general_2004} gives therefore that $(X^w_n,Z^w_n)$ is Harris ergodic.

If $w{<}0$, the last two steps can be done in a similar way.
In this case, the process $(-X^w(t))$ satisfies an analogous equation with the difference that the process ${\cal N}_{\beta,X^w}$ does not jump when ${-}X^w(t){>}c_\beta$ since $\beta(x){=}0$ for $x{\le}{-}c_\beta$.

\subsection{Characterization of $\Pi_w$}

Let $\widehat{\Pi}_w$ be the invariant probability distribution of $(X^w_n,Z^w_n)$.
With the above notations,
\[
\E_{\widehat{\Pi}_w}\left[\min(s_1,t_1)\right]\leq \E_{\widehat{\Pi}_w}\left[s_1\right]{=}\frac{1}{\lambda}<{+}\infty,
\]
the probability defined by the classical cycle formula,
\[
\frac{1}{\E_{\widehat{\Pi}_w}\left[\min(s_1,t_1)\right]}\E_{\widehat{\Pi}_w}\left[\int_0^{\min(s_1,t_1)}f(X^w(u),Z^w(u))\,\diff u\right],
\]
for any bounded Borelian function on $\R{\times}\R_+^\ell$ is an invariant distribution for the process $(X^w(t),Z^w(t))$.

Proposition~9.2 of~\citet{ethier_markov_2009} shows that any  distribution is invariant for $(X^w(t),Z^w(t))$ if and only if it satisfies Relation~\eqref{muOm}.
It remains to prove the uniqueness of the invariant distribution, using the fact that the embedded Markov chain has a unique invariant distribution.

Although this  is a natural  result, we have not  been able to  find a
reference in the literature.  Most results are stated for discrete
time, the continuous time is usually treated by looking at
the process on  a ``discrete skeleton'', i.e. at  instants multiple of
some  positive constant.  See
Proposition~3.8  of~\citet{asmussen_applied_2003}  for  example.
As this technique is not adapted to our system, we derive a different proof using the Palm measure of the associated stationary point process.

If $\mu$ is some invariant distribution of the Markov process $(X^w(t),Z^w(t))$,  we build a stationary version $((X^w(t),Z^w(t)), t{\in}\R)$ of it on the whole real line. In particular, we have that $(X^w(t),Z^w(t)){\steq{dist}}\mu$, for all $t{\in}\R$.

We denote by $(S_n,n{\in}\Z)$ the non-decreasing sequence of the jumps (due to ${\cal N}_\lambda$ and ${\cal N}_{\beta,X^w}$), with the convention $S_0{\le}0{<}S_1$
The sequence  $((X^w(S_n),Z^w(S_n)), n{\ge}0)$ has the same distribution as  the process $((X^w_n,Z^w_n), n{\ge}0)$, the Markov chain with initial state $(X^w(S_0),Z^w(S_0))$.
Since, for any $t{\in}\R$,
\[
\left(\rule{0mm}{4mm}(X^w(s{+}t),Z^w(s{+}t)), s{\in}\R\right)\steq{dist}\left(\rule{0mm}{4mm}(X^w(s),Z^w(s)), s{\in}\R\right),
\]
the marked point process ${\cal T}{\steq{def}}\left(\rule{0mm}{4mm}S_n,(X^w(S_n),Z^w(S_n)\right), n{\in}\Z)$ is a stationary point process, i.e.
  \[
  ((S_n,X^w(S_n),Z^w(S_n), n{\in}\Z) \steq{dist} ((S_n{-}t,X^w(S_n),Z^w(S_n), n{\in}\Z), \quad \forall t{\in}\R.
  \]
The {\em Palm measure} of ${\cal T}$ is a probability distribution $\widehat{Q}$ such that the sequence \break $((S_n{-}S_{n-1},X^w(S_n),Z^w(S_n), n{\in}\Z)$ is stationary.   See Chapter~11 of~\citet{robert_stochastic_2003} for a quick presentation of stationary point processes and Palm measures.
  
 Under $\widehat{Q}$, the Markov chain $((X^w(S_n),Z^w(S_n)), n{\ge}0)$ is at equilibrium.
 Using Harris ergodicity, we have proved in the previous section that the Markov chain $((X^w_n,Z^w_n), n{\ge}0)$ has a unique invariant measure.
 Considering that both sequences $((X^w_n,Z^w_n), n{\ge}0)$ and $((X^w(S_n),Z^w(S_n)), n{\ge}0)$ have the same distribution, we have that $\widehat{Q}\left( \R_+^{\Z}, \cdot\right)$ is uniquely determined.

Moreover, remembering that,
\[
    \widehat{Q}\left(S_n{-}S_{n-1}{>}t\right) = \E_{\widehat{Q}}\left[ \exp\left( {-}\int_{0}^t \beta\left(X^w(S_{n-1})e^{-s}\right)\diff s \right) \right]
\]
We have that $\widehat{Q}$ is entirely determined by the ergodic distribution of the embedded Markov chain and consequently that the Palm measure $\widehat{Q}$ is unique.
By Proposition~11.5 of~\citet{robert_stochastic_2003}, the distribution of ${\cal T}$ is expressed with $\widehat{Q}$.

We have, for every bounded function $f$,
\[
    \E_{\mu}\left[ f(X^w(0),Z^w(0)) \right] = \E_{\cal T}\left[f(X^w(S_0)e^{S_0},Z^w(S_0){\odot}e^{\gamma S_0})\right],
\]
which uniquely determines the invariant distribution $\mu$.

 The proposition is proved.
\end{proof}
\section{Averaging Principles for Discrete Models of Plasticity}
\label{app:alternative}

In this section, we present a general discrete model of plasticity, state the associated averaging principle theorem and give a sketch of its proof.
We will only point out the differences with the proof of  the main result of  this paper, Theorem~\ref{theorem:homog}.

For this model of plasticity, the membrane potential $X$, the plasticity processes $Z$ and the synaptic weight $W$ are integer-valued variables.
This system is illustrated in Section 7 of~\citet{robert_mathematical_2020} for calcium-based models.
It amounts to represent these three quantities $X$, $Z$ and $W$ as multiple of a ``quantum'' instead of a continuous variable.
The leaking mechanism in particular, the term corresponding to ${-}\gamma Y(t)\diff t$ in the continuous model, $Y{\in}\{X,Z,W\}$ and $\gamma{>}0$, in the SDEs, is represented by the fact that each quantum leaves the system at a fixed rate $\gamma$.

The main advantage of this model is that simple analytical expressions of the invariant distribution are available.

\begin{definition}\label{def:discretemodel}
The SDEs for the discrete model are
\begin{equation}
\label{eq:discretemodel}
\begin{cases}
\quad \diff X(t) &= \displaystyle{-}\sum_{i=1}^{X(t-)}\mathcal{N}_{1,i}(\diff t) +W(t{-})\mathcal{N}_{\lambda}(\diff t)-\sum_{i=1}^{X(t-)}\mathcal{N}_{\beta,i}(\diff t),\\
\quad \diff Z(t) &= \displaystyle -\sum_{i=1}^{Z(t-)}\mathcal{N}_{\gamma,i}(\diff t)+ B_1\mathcal{N}_{\lambda}(\diff t)+B_2\sum_{i=1}^{X(t-)}\mathcal{N}_{\beta,i}(\diff t),\\
\quad \diff \Omega_a(t)&=\displaystyle {-}\alpha\Omega_a(t)\diff t{+}n_{a,0}(Z(t))\diff t\\&\hspace{0cm}{+}\displaystyle
   n_{a,1}(Z(t{-}))\mathcal{N}_{\lambda}(\diff t){+}n_{a,2}(Z(t{-}))\sum_{i=1}^{X(t-)}\mathcal{N}_{\beta,i}(\diff t),\quad a{\in}\{p,d\},\\
\quad \diff W(t) &=\displaystyle -\sum_{i=1}^{W(t-)}\mathcal{N}_{\delta,i}(\diff t)\\
&\hspace{0cm}{+}A_p{\cal N}_{I,\Omega_{p}}(\diff t)-A_d\ind{W(t-){\ge}A_d}{\cal N}_{I,\Omega_{d}}(\diff t),
\end{cases}
\end{equation}
where $\beta,\gamma,\delta$ are non-negative real numbers, $B_1$, $B_2{\in}\N^{\ell}$ and, for $a{\in}\{p,d\}$, $A_a{\in}\N$. 
The functions $n_{a,i}$ are assumed to be bounded  by $C_n$.
\end{definition}
For $a{\in}\{p,d\}$, the function $I$ of ${\cal N}_{I,\Omega_{a}}$, defined in Section~\ref{secsec:interacting}, is the identity function $I(x){=}x$, $x{\in}\R$.  For $\xi{>}0$, ${\cal N}_\xi$, resp. $({\cal N}_{\xi,i})$, is a Poisson process on $\R_+$ with rate $\xi$, resp. an i.i.d. sequence of such point processes. All Poisson processes are assumed to be independent.
\begin{definition}\label{FVD}
For a fixed $w$, the process of the fast variables $(X^w(t),Z^w(t))$ on $\N{\times}\N^{\ell}$ of the SDEs is the Markov process whose transition rates are given by, for $(x,z){\in}\N{\times}\N^{\ell}$,
\[
(x,z)\longrightarrow
\begin{cases}
\hspace{1mm}(x{+}w,z{+}B_1) & \lambda, \\
\hspace{1mm}(x{-}1,z)   & x,
\end{cases}
\hspace{2cm}
\longrightarrow
\begin{cases}
\hspace{1mm}(x,z{-}1)   & \gamma z, \\
\hspace{1mm}(x{-}1,z{+}B_2)  & \beta x.
\end{cases}
\]
\end{definition}

\begin{theorem}[Averaging Principle for a Discrete Model] \label{th:StochD}
If the assumptions of Definition~\ref{def:discretemodel} are verified, the family of scaled processes $(W_{\eps}(t))$ associated to Relations~\eqref{eq:discretemodel} is converging in distribution, as $\eps$ goes to $0$, to the \cadlag integer-valued process $(w(t))$ satisfying the ODE
\begin{equation}\label{eqdiscreteSAPD}
\quad \diff w(t) = -\sum_{i=1}^{w(t-)}\mathcal{N}_{\gamma,i}(\diff t) {+}A_p{\cal N}_{I,\omega_{p}}(\diff t){-}A_d\ind{w(t{-}){\ge}A_d}{\cal N}_{I,\omega_{d}}(\diff t),
\end{equation}
and, for $a{\in}\{p,d\}$,
\[
\frac{\diff \omega_a}{\diff t}(t)= {-}\alpha  \omega_a(t){+}\\
    \int_{\N{\times}\N^\ell} \left(\rule{0mm}{3mm} n_{a,0}(z){+}\lambda n_{a,1}(z){+}\beta(x) n_{a,2}(z)\right)\Pi_{w(t)}(\diff x,\diff z),
\]
where $\Pi_{w}$ is the invariant distribution of the Markov process of Definition~\ref{FVD}.
\end{theorem}

\begin{proof}
Again, we have to show that,  on a fixed finite interval, the  process $(W(t))$ is bounded with high probability.
A coupled process that stochastically bounds from above the discrete process is also defined.

\begin{definition}\label{def:discretemodelcoupled}
The process $(\Dom{X}(t), \Dom{Z}(t), \Dom{\Omega}(t), \Dom{W}(t) )$ satisfies the following SDEs
\begin{equation}
\label{eq:discretemodelcoupled}
\begin{cases}
\quad \diff \Dom{X}(t) &= \displaystyle{-}\sum_{i=1}^{\Dom{X}(t-)}\mathcal{N}_{1,i}(\diff t) +\Dom{W}(t{-})\mathcal{N}_{\lambda}(\diff t)-\sum_{i=1}^{\Dom{X}(t-)}\mathcal{N}_{\beta,i}(\diff t),\\
\quad \diff \Dom{Z}(t) &= \displaystyle -\sum_{i=1}^{\Dom{Z}(t-)}\mathcal{N}_{\gamma,i}(\diff t)+ B_1\mathcal{N}_{\lambda}(\diff t)+B_2\sum_{i=1}^{\Dom{X}(t-)}\mathcal{N}_{\beta,i}(\diff t),\\
\quad \diff \Dom{\Omega}(t)&=\displaystyle {-}\alpha\Dom{\Omega}(t)\diff t{+}C_n\diff t\\&\hspace{0cm}{+}\displaystyle
   C_n\mathcal{N}_{\lambda}(\diff t){+}C_n\sum_{i=1}^{\Dom{X}(t-)}\mathcal{N}_{\beta,i}(\diff t),\\
\quad \diff \Dom{W}(t) &=\displaystyle A_p{\cal N}_{I,\Dom{\Omega}}(\diff t),
\end{cases}
\end{equation}
where  $B_1$, $B_2{\in}\N^{\ell}$ and, for $a{\in}\{p,d\}$, $A_p{\in}\N$.
\end{definition}
It is not difficult to prove that this process is indeed  a coupling that verifies the relation $W(t){\leq}\Dom{W}(t)$, for all $t{\ge}0$ and that the process $(\Dom{W}(t))$ is non-decreasing.

From the SDEs governing the scaled version of the coupled system, we obtain
\begin{align*}
    &\E\left[\Dom{W}_{\eps}(t){-}w_0\right]\leq A_p\E\left[\int_0^t{\cal N}_{I,\Dom{\Omega}_{\eps,p}}\diff s\right]{\le}A_p t \E\left[\sup_{s\leq t}\Dom{\Omega}_{\eps,p}(s)\right]\\
    &\quad\leq A_p t \E\left[\omega_{0}{+}C_n\sup_{s\le t}\int_0^se^{-\alpha (s{-}u)}\left(\diff u + \eps\mathcal{N}_{\lambda/\eps}(\diff u){+}\eps\sum_{i=1}^{\Dom{X}(u-)}\mathcal{N}_{\beta/\eps,i}(\diff u)\right)\right]\\
    &\leq A_p t \left(\omega_{0}{+}\frac{C_n}{\alpha}(1{+}\lambda){+}\E\left[\int_0^t\beta\Dom{X}_{\eps}(u)\right]\diff u\right)\\
    &\leq A_p t \left(\omega_{0}{+}\frac{C_n}{\alpha}(1{+}\lambda){+}\lambda \beta \E\left[\int_0^t\Dom{W}_{\eps}(u)\diff u\right]\right)
\leq D {+}D\int_0^t\E\left[\Dom{W}_{\eps}(u)\right]\diff u,
\end{align*}
for all $t{\le}T$, for some constant $D{\ge}0$.
Gronwall's Lemma gives a uniform bound, with respect to $\eps$, 
\[
\E\left[\sup_{t{\le}T}\Dom{W}_{\eps}(t)\right] =\E\left[\Dom{W}_{\eps}(T)\right])\le(D{+}w_0)e^{D T}.
\]
Using Markov inequality, we have then that, for any $\eta{>}0$, the existence of $K_0$ and $n_0$ such that $n{\ge}n_0$, the inequality
\[
\P\left(\sup_{t{\le}T} \Dom{W}_{\eps_n}(t){\le} K_0\right)\ge 1{-}\eta
\]
holds.
We can then finish the proof in the same way as in Section~\ref{secsec:theoremproof}.
The tightness property of the family of \cadlag processes $(\Dom{W}_{\eps}(t))$, $\eps{\in}(0,1)$ are proved with Aldous' criterion, see Theorem~VI.4.5 of~\citet{jacod_limit_1987}.

We have to prove the uniqueness of the solution of Relation~\eqref{eqdiscreteSAPD} and the convergence in distribution of the scaled process to the process $w(t)$.
For this, we need to have some Lipschitz property on the limiting system, and  finite first moments for the invariant distribution of $(X^w(t),Z^w(t))$.
This is proved in Section~7 of~\citet{robert_mathematical_2020} for the case where $Z$ is a one-dimensional process, the extension to multi-dimensional $Z$ is straightforward.
\end{proof}
\end{document}

%% file: Figure_Proof.pdf_tex
%% Creator: Inkscape 1.0.1 (c497b03c, 2020-09-10), www.inkscape.org
%% PDF/EPS/PS + LaTeX output extension by Johan Engelen, 2010
%% Accompanies image file 'Figure_Proof.pdf' (pdf, eps, ps)
%%
%% To include the image in your LaTeX document, write
%%   \input{<filename>.pdf_tex}
%%  instead of
%%   \includegraphics{<filename>.pdf}
%% To scale the image, write
%%   \def\svgwidth{<desired width>}
%%   \input{<filename>.pdf_tex}
%%  instead of
%%   \includegraphics[width=<desired width>]{<filename>.pdf}
%%
%% Images with a different path to the parent latex file can
%% be accessed with the `import' package (which may need to be
%% installed) using
%%   \usepackage{import}
%% in the preamble, and then including the image with
%%   \import{<path to file>}{<filename>.pdf_tex}
%% Alternatively, one can specify
%%   \graphicspath{{<path to file>/}}
%% 
%% For more information, please see info/svg-inkscape on CTAN:
%%   http://tug.ctan.org/tex-archive/info/svg-inkscape
%%
\begingroup%
  \makeatletter%
  \providecommand\color[2][]{%
    \errmessage{(Inkscape) Color is used for the text in Inkscape, but the package 'color.sty' is not loaded}%
    \renewcommand\color[2][]{}%
  }%
  \providecommand\transparent[1]{%
    \errmessage{(Inkscape) Transparency is used (non-zero) for the text in Inkscape, but the package 'transparent.sty' is not loaded}%
    \renewcommand\transparent[1]{}%
  }%
  \providecommand\rotatebox[2]{#2}%
  \newcommand*\fsize{\dimexpr\f@size pt\relax}%
  \newcommand*\lineheight[1]{\fontsize{\fsize}{#1\fsize}\selectfont}%
  \ifx\svgwidth\undefined%
    \setlength{\unitlength}{288.67695245bp}%
    \ifx\svgscale\undefined%
      \relax%
    \else%
      \setlength{\unitlength}{\unitlength * \real{\svgscale}}%
    \fi%
  \else%
    \setlength{\unitlength}{\svgwidth}%
  \fi%
  \global\let\svgwidth\undefined%
  \global\let\svgscale\undefined%
  \makeatother%
  \begin{picture}(1,1.26157778)%
    \lineheight{1}%
    \setlength\tabcolsep{0pt}%
    \put(0,0){\includegraphics[width=\unitlength]{Figure_Proof.pdf}}%
    \put(0.3263125,0.32943912){\color[rgb]{0,0,0}\makebox(0,0)[lt]{\smash{\begin{tabular}[t]{l}\shortstack{\textbf{Truncated process $\Dom{U}^K$}\\ Scaled system $\Dom{U}^K_{\eps}$ Eq.\eqref{SDELinTrunc} \\ Tightness in Proposition \ref{TightBoundProp}\\ SAP in Proposition \ref{HomTruncProp}}\end{tabular}}}}%
    \put(0.33150859,0.61317249){\color[rgb]{0,0,0}\makebox(0,0)[lt]{\smash{\begin{tabular}[t]{l}\shortstack{\textbf{Coupled process $\Dom{U}$}\\ System of SDE  Eq.\eqref{eq:markovmaj}\\ Scaled system $\Dom{U}_{\eps}$ Eq.\eqref{SDELin} \\ SAP + Convergence\\  in Proposition \ref{prop:sapdom}\\ Limit $(\Dom{\omega},\Dom{w})$ in Eq.\eqref{AsymLinODE}}\end{tabular}}}}%
    \put(0.32631244,0.96965149){\color[rgb]{0,0,0}\makebox(0,0)[lt]{\smash{\begin{tabular}[t]{l}\shortstack{\textbf{Stochastic process $U$}\\ System of SDE  Eq.\eqref{eq:markov}\\ Scaled system $U_{\eps}$ Eq.\eqref{eq:markW} \\ SAP in Theorem \ref{theorem:homog}\\ Limit $(\omega_p,\omega_d,w)$ in Eq.\eqref{ODEH}}\end{tabular}}}}%
    \put(0.74141612,0.82071966){\color[rgb]{0,0,0}\makebox(0,0)[lt]{\smash{\begin{tabular}[t]{l}\shortstack{\textbf{Section \ref{section:coupling}}\\ Coupling property \\ in Proposition \ref{coupprop} \\ $|W(t)| {\leq}  \Dom{W}(t)$}\end{tabular}}}}%
    \put(0.78729195,0.52121169){\color[rgb]{0,0,0}\makebox(0,0)[lt]{\smash{\begin{tabular}[t]{l}\shortstack{\textbf{Section \ref{OccSec:1}}\\ $\Dom{W}^K(t) {\le} K$}\end{tabular}}}}%
    \put(0.01328304,0.49741916){\color[rgb]{0,0,0}\makebox(0,0)[lt]{\smash{\begin{tabular}[t]{l}\shortstack{\textbf{Section \ref{section:proof-1}}\\ Monotonicity + \\ Analytical ODE +\\ Uniqueness of limit}\end{tabular}}}}%
    \put(0.04965595,0.82391719){\color[rgb]{0,0,0}\makebox(0,0)[lt]{\smash{\begin{tabular}[t]{l}\shortstack{\textbf{Section \ref{secsec:theoremproof}}\\ Coupling  $\rightarrow$\\ $W_{\eps}$ Uniformly \\ Bounded}\end{tabular}}}}%
    \put(0.30956979,0.05334084){\color[rgb]{0,0,0}\makebox(0,0)[lt]{\smash{\begin{tabular}[t]{l}\shortstack{\textbf{Section \ref{OccSec:2}}\\ $\rightarrow$ Tightness\\ \textbf{Section \ref{OccSec:3}} + \textbf{Appendix \ref{sec:aproof}}\\ $\rightarrow$ SAP}\end{tabular}}}}%
    \put(0.77561266,0.0661018){\color[rgb]{0,0,0}\makebox(0,0)[lt]{\smash{\begin{tabular}[t]{l}\shortstack{\textbf{Appendix \ref{ShotSec}}\\ Shot-Noise\\Processes}\end{tabular}}}}%
    \put(0.04905041,1.13061046){\color[rgb]{0,0,0}\makebox(0,0)[lt]{\smash{\begin{tabular}[t]{l}\shortstack{\textbf{Appendix \ref{section:invariant}}\\ Existence and\\ Unicity $\Pi_w$}\end{tabular}}}}%
  \end{picture}%
\endgroup%